\documentclass[12pt]{amsart}
\usepackage[all]{xy}
\usepackage{a4wide}
\usepackage{mathtools}
\usepackage{amssymb}
\usepackage{amsthm}
\usepackage{amsmath}
\usepackage{amscd}
\usepackage{verbatim}
\usepackage[all]{xy}

\usepackage{amsmath,amsthm,amssymb}
\usepackage{enumitem}

\usepackage{float}
\usepackage{ifthen}
\usepackage{xargs}
\usepackage{url}
\usepackage[all]{xy}
\usepackage{xcolor}
\usepackage{comment}
\usepackage[obeyDraft,textsize=tiny]{todonotes}
\specialcomment{draftnote}{\vspace{1em}\begingroup\ttfamily\footnotesize\color{blue}\begin{minipage}{15cm}}{\end{minipage}\endgroup\vspace{1em}}
\specialcomment{alert}{\vspace{1em}\begingroup\ttfamily\footnotesize\color{blue}\begin{minipage}{15cm}}{\end{minipage}\endgroup\vspace{1em}}

\usepackage[refpage,intoc,noprefix]{nomencl}

\newcommand{\R}{\mathbb{R}} 
\newcommand{\Q}{\mathbb{Q}} 
\newcommand{\C}{\mathbb{C}} 
\newcommand{\Z}{\mathbb{Z}} 
\newcommand{\h}{\mathbb{H}} 

\newcommand*{\reg}{\mathrm{reg}}

\newcommand{\bs}{\backslash}

\newcommand{\mt}{\mathbf{t}}

\DeclareMathOperator{\Iso}{Iso}

\newcommand{\IM}{I^{\mathrm{M}}}
\newcommand{\ISh}{I^{\mathrm{Sh}}}
\newcommand{\LambdaM}{\Lambda^{\mathrm{M}}}

\newcommand{\SL}{{\text {\rm SL}}}
\newcommand{\G}{\Gamma}
\newcommand{\dg}{\mathcal{D}} 
\newcommand{\dgdelta}{{\mathcal{D}(\Delta)}}
\newcommand{\abs}[1]{\left\vert#1\right\vert}
\newcommand{\e}{\mathfrak{e}}
\newcommand{\Deltaover}[1]{\left(\frac{\Delta}{#1}\right)}

\newcommand{\smallabcd}{\left(\begin{smallmatrix}a & b \\ c & d\end{smallmatrix}\right)}

\newcommand{\smallTmatrix}{\left(\begin{smallmatrix}1 & 1 \\ 0 & 1\end{smallmatrix}\right)}
\newcommand{\smallSmatrix}{\left(\begin{smallmatrix}0 & -1 \\ 1 & 0\end{smallmatrix}\right)}

\DeclareMathOperator{\PSL}{PSL}
\DeclareMathOperator{\Mp}{Mp}

\DeclareMathOperator{\sgn}{sgn}

\newtheorem{theorem}{Theorem}[section]
\newtheorem{proposition}[theorem]{Proposition}
\newtheorem{lemma}[theorem]{Lemma}
\newtheorem{corollary}[theorem]{Corollary}

\theoremstyle{definition}

\newtheorem{example}[theorem]{Example}
\newtheorem{remark}[theorem]{Remark}

\numberwithin{theorem}{section} \numberwithin{equation}{section}

\DeclareMathOperator{\erfc}{erfc}

\begin{document}

\title[]{On a theta lift related to the Shintani lift}

\author{Claudia Alfes-Neumann and Markus Schwagenscheidt}
\maketitle

\begin{abstract} We study a certain theta lift which maps weight $-2k$ to weight $1/2-k$ harmonic weak Maass forms for $k \in \Z, k \geq 0$, and which is closely related to the classical Shintani lift from weight $2k+2$ to weight $k+3/2$ cusp forms. We compute the Fourier expansion of the theta lift and show that it involves twisted traces of CM values and geodesic cycle integrals of the input function. As an application, we obtain a criterion for the non-vanishing of the central $L$-value of an integral weight newform $G$ in terms of the holomorphicity of the theta lift of a certain harmonic weak Maass form associated to $G$. Moreover, we derive interesting identities between cycle integrals of different kinds of modular forms.
\end{abstract}

\section{Introduction}

A famous result of Zagier \cite{zatr} states that the twisted traces of singular moduli, i.e. the values of the modular $j$-invariant at quadratic irrationalities in the upper half-plane, occur as the Fourier coefficients of weakly holomorphic modular forms of weight $1/2$ and $3/2$. Bruinier and Funke \cite{brfu06} showed that the generating series of the traces of singular moduli can be obtained as the image of a certain theta lift of $J = j - 744$. Using this approach, new proofs of Zagier's results, including the modularity of generating series of twisted traces of singular moduli, and generalizations to higher weight and level have been studied in several recent works, e.g. \cite{ae}, \cite{brono2}, \cite{alfes}, \cite{AGOR}. 
For example, in \cite{AGOR} a twisted theta lift from weight $0$ to weight $1/2$ harmonic weak Maass forms was defined which allowed to recover Zagier's generating series of weight $1/2$ as a theta lift. Further, it turned out that this lift is closely related to the Shintani lift via the $\xi$-operator on harmonic weak Maass forms. The classical Shintani lift establishes a connection between integral and half-integral modular forms \cite{shintani} and is an indispensable tool in the theory of modular forms. 
 Using this relationship between integral and half-integral weight modular forms a number of remarkable theorems were proven, for example the famous theorem of Waldspurger \cite{Waldspurger}, which asserts that the central critical value of the twisted $L$-function of an even weight newform is proportional to the square of a coefficient of a half-integral weight modular form.

In \cite{AGOR}, the connection between the two lifts led to a more explicit version of a theorem of Bruinier and Ono \cite{brono} which states that the vanishing of the central derivative of the $L$-series of an elliptic curve is determined by the algebraicity of a Fourier coefficient of the holomorphic part of a certain harmonic weak Maass form of weight $1/2$. In the present work, we study a generalization of the theta lift considered in \cite{AGOR}, which we call the Millson theta lift. Our lift maps weight $-2k$ to weight $1/2-k$ harmonic weak Maass forms, where $k \in \Z_{\geq 0}$, and is again related to the Shintani lift via the $\xi$-operator. We completely determine the Fourier expansion of the Millson lift of a harmonic weak Maass form $F$ of weight $-2k$, and we show that the coefficients of the holomorphic part of the lift are given by twisted traces of CM-values of the weight $0$ form $R_{-2k}^{k}F$, whereas the coefficients of the non-holomorphic part are given by twisted traces of geodesic cycle integrals of the weight $2k+2$ cusp form $\xi_{-2k}F$. Additionally, inspired by the relation of the Millson lift to the Shintani lift, we prove interesting identities between cycle integrals of $\xi_{-2k}F$ and $R_{-2k}^{2j+1}F$, with varying $j\geq 0$, for a harmonic weak Maass form $F$ of weight $-2k$. In certain cases the cycle integrals of $R_{-2k}^{2j+1}F$ do not converge, and we propose a regularization in these cases.

The necessary computations are quite involved due to the very general setup, but we believe that they will be very useful for the study of similar theta lifts in the future. Further, we hope that our lift can be used to prove a higher weight version of the aforementioned theorem of Bruinier and Ono \cite{brono} on the non-vanishing of the central values of derivatives of $L$-functions of even weight newforms.



To illustrate our results, let us simplify the setup by restricting to modular forms for the full modular group $\SL_{2}(\Z)$. In the body of the paper we also treat forms for arbitrary congruence subgroups by using the theory of vector valued modular forms for the Weil representation of an even lattice of signature $(1,2)$.

We let $z = x+iy \in \h$ and $q=e^{2\pi i z}$. Recall from \cite{brfu04} that a harmonic weak Maass form of weight $k \in \Z$ is a smooth function $F:\h \to \C$ which is invariant under the usual weight $k$ slash operation of $\SL_2(\Z)$, which is annihilated by the weight $k$ hyperbolic Laplace operator $\Delta_{k}$, and for which there is a Fourier polynomial $P_{F}=\sum_{n\leq 0} a^{+}(n)q^{n}\in\C[q^{-1}]$ such that $F - P_{F}$ is rapidly decreasing at $i\infty$.
The space of such forms is denoted by $H_{k}^{+}$. Every $F \in H_{k}^{+}$ has a Fourier expansion consisting of a holomorphic part $F^{+}$ and a non-holomorphic part $F^{-}$,
\begin{align*}
F(z) = F^{+}(z) + F^{-}(z) = \sum_{n \gg -\infty}a^{+}(n)q^{n} + \sum_{n < 0}a^{-}(n)\Gamma(1-k,4\pi |n|y)q^{n},
\end{align*} 
where $\Gamma(s,x) = \int_{x}^{\infty}t^{s-1}e^{-t}dt$ is the incomplete Gamma function.
 Harmonic weak Maass forms of half-integral weight for $\Gamma_{0}(4)$ are defined analogously. Important tools in the theory of harmonic weak Maass forms are the Maass lowering and raising operators $L_{k} = -2i\frac{\partial}{\partial \bar{z}}$ and $R_{k} = 2iy^{2}\frac{\partial}{\partial z}+ky^{-1}$, which lower or raise the weight of a real analytic modular form by $2$, as well as the surjective antilinear differential operator $\xi_{k}: H_{k}^{+} \to S_{2-k}$ defined by $\xi_{k}F(z) = 2iy^{k}\overline{\frac{\partial}{\partial \bar{z}}F(z)}$.

 Let $D \in \Z$ be a discriminant. We let $\mathcal{Q}_{D}$ be the set of integral binary quadratic forms $[a,b,c] = ax^{2} + bxy + cy^{2}$ of discriminant $b^{2}-4ac = D$. The modular group $\SL_{2}(\Z)$ acts on $\mathcal{Q}_{D}$ from the right, with finitely many classes if $D \neq 0$. For $D < 0$ we can split $\mathcal{Q}_{D} = \mathcal{Q}_{D}^{+} \sqcup \mathcal{Q}_{D}^{-}$ into the subsets of positive definite ($a > 0$) and negative definite ($a < 0$) forms. Further, for $D < 0$ the stabilizer $\overline{\SL_2(\Z)}_{Q}$ of $Q \in \mathcal{Q}_{D}$ in $\PSL_2(\Z)$ is finite, and for $D > 0$ the stabilizer $\overline{\SL_2(\Z)}_{Q}$ is trivial if $D$ is a square and infinite cyclic otherwise.
 
 Let $Q = [a,b,c] \in \mathcal{Q}_{D}$. For $D < 0$ there is an associated CM point $\alpha_{Q} = (-b+i\sqrt{|D|})/2a\in \h$, while for $D > 0$ the solutions of $a|z|^{2} + b\Re(z) + c = 0$ define a geodesic $c_{Q}$ in $\h$, which is equipped with a certain orientation.

Let $\Delta \in \Z$ be a fundamental discriminant (possibly $1$). For $k \in \Z_{\geq 0}$ with $(-1)^{k}\Delta < 0$ the $\Delta$-th Shintani lift of a cusp form $F \in S_{2k+2}$ is (in our normalization) defined by
\[
\ISh_{\Delta}(F,\tau) = -|\Delta|^{-(k+1)/2}\sum_{\substack{d > 0 \\ (-1)^{k+1}d \equiv 0,1(4)}}\sum_{ \mathcal{Q}_{d|\Delta|}/\SL_2(\Z)}\chi_{\Delta}(Q)\mathcal{C}(F,Q) \  e^{2\pi i d\tau},
\] 
where 
\[
\mathcal{C}(F,Q) = \int_{\SL_2(\Z)_{Q}\setminus c_{Q}}F(z)Q(z,1)^{k}dz
\]
is a geodesic cycle integral of $F$ and
\begin{equation*}
\chi_{\Delta}(Q)=
\begin{cases}
\Deltaover{n}, & \text{if } (a,b,c,\Delta) = 1 \text{ and $Q$ represents $n$ with $(n,\Delta) = 1$,}
\\
0, &\text{otherwise},
\end{cases}
\end{equation*}
is a genus character. It is well known that the Shintani lift of $F$ is a cusp form in $S_{k+3/2}(\Gamma_{0}(4))$ which satisfies the Kohnen plus space condition, i.e.~the $d$-th Fourier coefficient vanishes unless $(-1)^{k+1}d \equiv 0,1 (4)$.

For a $\SL_2(\Z)$-invariant function $F$ and $d < 0$ we define twisted traces by
\[
\mt_{\Delta}^{+}(F;d) = \sum_{Q^{+}_{d|\Delta|}/\SL_2(\Z)}\chi_{\Delta}(Q)\frac{F(\alpha_{Q})}{|\overline{\SL_2(\Z)}_{Q}|}, \qquad \mt_{\Delta}^{-}(F;d) = \sum_{Q^{-}_{d|\Delta|}/\SL_2(\Z)}\chi_{\Delta}(Q)\frac{F(\alpha_{Q})}{|\overline{\SL_2(\Z)}_{Q}|},
\]
and for $d > 0$ and a function $F$ transforming of weight $2k+2$ for $\SL_2(\Z)$ we define
\[
\mt_{\Delta}(F;d) = \sum_{Q_{d|\Delta|}/\SL_2(\Z)}\chi_{\Delta}(Q)\mathcal{C}(F,Q),
\]
whenever the cycle integrals $\mathcal{C}(F,Q)$ converge.

Let $F \in H_{-2k}^{+}$ be a harmonic weak Maass form. We define the Millson theta lift by
\begin{align}\label{ThetaLiftIntroduction}
\IM_{\Delta}(F,\tau) = \int_{\SL_2(\Z) \setminus \h}^{\text{reg}}F(z)\Theta_{\Delta}(\tau,z,\psi_{M,k})y^{-2k-2}dx \, dy,
\end{align}
where $\Theta_{\Delta}(\tau,z,\psi_{M,k})$ is the twisted theta function associated to a certain Schwartz function $\psi_{M,k}$, and the integral has to be regularized in a certain way to ensure convergence. The theta function, and thus also $\IM_{\Delta}(F,\tau)$, transforms like a modular form of weight $1/2-k$ in $\tau$. We are now ready to state our main result for the Millson theta lift for $\SL_{2}(\Z)$.

\begin{theorem}\label{MainTheoremIntroduction}
	Let $k \in \Z_{\geq 0}$ such that $(-1)^{k}\Delta < 0$ and let $F \in H_{-2k}^{+}$ with vanishing constant term $a^{+}(0)$. 
	\begin{enumerate}
		\item The Millson theta lift $\IM_{\Delta}(F,\tau)$ is a harmonic weak Maass form in $H_{1/2-k}^{+}(\Gamma_{0}(4))$ satisfying the Kohnen plus space condition. Further, if $F$ is weakly holomorphic, then so is $\IM_{\Delta}(F,\tau)$. 
		\item $\IM_{\Delta}(F,\tau)$ is related to the Shintani lift of $\xi_{-2k,z}F \in S_{2k+2}$ by
	\begin{align*}
	\xi_{1/2-k,\tau}\IM_{\Delta}(F,\tau) = -\frac{\sqrt{|\Delta|}}{2}\ISh_{\Delta}(\xi_{-2k,z}F,\tau).
	\end{align*}
		\item The Fourier expansion of $\IM_{\Delta}(F,\tau)$ is given by
	\begin{align*}
	\IM_{\Delta}(F,\tau) &= \sum_{d > 0}\frac{2}{\sqrt{d}}\left( \frac{1}{2\pi\sqrt{|\Delta|d}}\right)^{k}\mt_{\Delta}^{+}(R_{-2k}^{k}F;d)q^{d} \\
	&\quad-\sum_{b > 0}2i\epsilon\left( \frac{1}{2\pi i |\Delta|}\right)^{k}\sum_{\substack{n < 0}}\left( \frac{\Delta}{n}\right)a^{+}(nb)(4\pi n)^{k}q^{-|\Delta|b^{2}} \\
	&\quad - \sum_{d < 0}\frac{1}{2(\pi|d|)^{k+1/2}|\Delta|^{k/2}}\overline{\mt_{\Delta}(\xi_{-2k}F;d)}\Gamma(\tfrac{1}{2}+k,4\pi|d|v)q^{d},
	\end{align*}
	where $R_{-2k}^{k} = R_{-2}\circ R_{-4}\circ \dots \circ R_{-2k}$ is the iterated Maass raising operator and $\epsilon$ equals $1$ or $i$ according to whether $\Delta > 0$ or $\Delta < 0$.
	\end{enumerate}
\end{theorem}

\begin{remark}
	\begin{enumerate}
		\item The assumption $a^{+}(0) = 0$ was imposed here to simplify the exposition in the introduction and will not be used in the body of the paper. If $a^{+}(0) \neq 0$ then $\IM_{\Delta}(F,\tau)$ also has a constant coefficient, and for $k = 0$ further non-holomorphic terms appear. In fact, for $k = 0$ the $\xi$-image of $\IM_{\Delta}(F,\tau)$ turns out to be a linear combination of unary theta series of weight $3/2$. Thus, using the theta lift, one can obtain formulas for the coefficients of mock theta functions of weight $1/2$ as traces of modular functions, similarly as in \cite{alfes}. The details will be discussed in a subsequent paper.	
		\item In Section \ref{Extensions} we extend the Millson theta lift to more general harmonic weak Maass forms whose non-holomorphic parts may also grow exponentially at the cusps. Further, we extend the Shintani lift to weakly holomorphic cusp forms and obtain a new proof of Theorem 1.3.(1) of \cite{briguka}, where the authors defined a generalized Shintani lift of a weakly holomorphic cusp form $F$ as the generating series of certain regularized cycle integrals of $F$.
		\item In \cite{briguka2}, the authors studied a so-called Zagier lift, which (for level $1$ and $k > 1$) maps weight $-2k$ to weight $1/2-k$ harmonic weak Maass forms. The proof of the modularity of this lift uses the Fourier coefficients of non-holomorphic Poincar\'e series together with the fact that a harmonic Maass form of negative weight is uniquely determined by its principal part. Thus their proof does not work for $k = 0$. In fact, the Zagier lift agrees with our lift in level $1$, so our theorem generalizes Proposition 6.2. of \cite{briguka2} to arbitrary level and to $k =0$, using a very different proof.
		\item Integrating $\Theta_{\Delta}(\tau,z,\psi_{M,k})$ in $\tau$ against a harmonic Maass form of weight $1/2-k$ yields a so-called locally harmonic Maass form of weight $-2k$. This lift was considered in \cite{hoevel}, \cite{brikavia} and \cite{crawford}, and it was shown that the resulting theta lift is related to the Shimura lift via the $\xi$-operator.
	\end{enumerate}
\end{remark}

For the proof of the theorem, we use the interpretation of the Shintani lift as a theta lift and employ various differential equations between the Millson and the Shintani theta functions. Further, we define a certain auxiliary theta lift constructed from the $k=0$ Millson theta function and suitable applications of iterated Maass raising and lowering operators, and show that this auxiliary lift essentially agrees with the Millson theta lift on harmonic weak Maass forms of weight $-2k$. This identity of theta lifts is a bit surprising and interesting in its own right, but due to its quite technical appearence we chose not to state it in the introduction but refer the reader to Theorem \ref{thm:relationlifts}.

The relation between the Millson lift and the Shintani lift also yields an interesting criterion on the vanishing of the twisted $L$-function of a newform at the critical point.
	\begin{theorem}
		Let $F \in H_{-2k}^{+}$, with vanishing constant terms at all cusps if $k = 0$, such that $G =\xi_{-2k}F \in S_{2k+2}$ is a normalized newform. For $(-1)^{k}\Delta < 0$ the lift $\IM_{\Delta}(F,\tau)$ is weakly holomorphic if and only if $L(G,\chi_{\Delta},k+1) = 0$.
	\end{theorem}
	
	\begin{remark}  For the general result regarding forms of higher level, see Theorem \ref{NonvanishingTheorem}. A version of this theorem for square-free level $N$ and odd weight $k$ has been proved in \cite{alfes}, Theorem 1.1., using the same techniques. Further, the above theorem in the case of level $1$ and $k > 0$ already appeared in \cite{briguka2}, Corollary 1.3, but the proof used very different arguments, i.e. the Zagier lift of non-holomorphic Poincar\'e series.
	\end{remark}

The Fourier coefficients of the non-holomorphic part of the Millson lift in Theorem \ref{MainTheoremIntroduction} involve cycle integrals of the cusp form $\xi_{-2k}F$, which reflects the relation between the Millson and the Shintani lift on the level of Fourier expansions. On the other hand, the fact that the Millson theta lift agrees up to some constant with a theta lift of the real-analytic modular function $R_{-2k}^{k}F$, see Theorem \ref{thm:relationlifts}, suggests that the Fourier coefficients of the non-holomorphic part of the Millson lift should also be expressible in terms of cycle integrals of $R_{-2k}^{k}F$. Inspired by this idea, we prove the following identities between the cycle integrals of different types of modular forms.

\begin{theorem}
	Let $D > 0$ be a discriminant which is not a square and let $Q \in \mathcal{Q}_{D}$. Let $k \in \Z_{\geq 0}$ and $F \in H_{-2k}^{+}$. For $j \in \Z_{\geq 0}$ we have
	\[
	\mathcal{C}(R_{-2k}^{2j+1}F,Q) = \frac{1}{D^{k-j}}\frac{j!(k-j)!(2k)!}{k!(2k-2j)!}\overline{\mathcal{C}(\xi_{-2k}F,Q)},
	\]
	where $R_{-2k}^{2j+1} = R_{-2k+2j}\circ \dots \circ R_{-2k}$ is the iterated Maass raising operator.
\end{theorem}

We prove this identity by a direct computation using Stokes' theorem and commutation relations for the differential operators involved. 

\begin{remark}
	\begin{enumerate}
		\item It is interesting to note that the cycle integral of $\xi_{-2k}F$ on the right-hand side does not depend on $j$. In particular, the cycle integrals of $R_{-2k}^{2j+1}F$ for different choices of $j$ are related by a very simple explicit constant.
 		\item The general result is given in Corollary \ref{CycleIntegralMainTheorem}. By plugging in special values for $j$, e.g. $j = k$, we obtain further interesting formulas (see Corollaries \ref{CycleIntegralTheorem} and \ref{CycleIntegralTheorem2}), which where previously given in Theorem 1.1. from \cite{briguka2} and Theorem 1.1. from \cite{briguka}. The above identity gives a unified proof for these two previously known, but seemingly unrelated results.
		\item We also define a regularized cycle integral $\mathcal{C}^{\reg}(R_{-2k}^{2j+1}F,Q)$ in the case that the discriminant of $Q$ is a square and the associated geodesic is infinite, and derive an analog of the above theorem in this situation, see Section \ref{InfiniteGeodesics}.
	\end{enumerate}
\end{remark}

The paper is organized as follows. In Section $2$ we introduce the basic setup and recall the necessary facts about vector valued harmonic weak Maass forms for the Weil representation associated with an even lattice of signature $(1,2)$. Sections $3$ and $4$ are devoted to the study of the (untwisted) Millson and Shintani theta functions and the properties of the corresponding theta lifts. In particular, the relation between the Millson and the Shintani theta lift is proven in Section $4$. The Fourier expansion of the untwisted Millson theta lift is computed in Section $5$. The necessary calculations are quite delicate and take up a major part of this work.
In Section $6$ we use a method from \cite{ae} to twist the results of the previous sections, i.e.~we derive the relation between the twisted Millson and Shintani theta functions and the Fourier expansion of the twisted lifts from the corresponding untwisted results. Finally, in Section $7$ we consider identities of cycle integrals of different types of modular forms.

\section*{Acknowledgements}

We thank Kathrin Bringmann, Jan Hendrik Bruinier, Stephan Ehlen, Jens Funke and Ben Kane for their help. 

The authors were partially supported by the DFG Research Unit FOR 1920 "Symmetry, Geomerty and Arithmetic``.

\section{Preliminaries}

For a positive rational number $N$ we consider the rational quadratic space of signature $(1,2)$ given by
\[
V=\left\{X=\begin{pmatrix} x_{2} & x_{1}\\ x_{3}& -x_{2}\end{pmatrix}; x_{1},x_{2},x_{3} \in \Q\right\}
\]
with the quadratic form $Q(X)=N\text{det}(X)$.
The associated bilinear form is $(X,Y)=-N\text{tr}(XY)$ for $X,Y \in V$.
The group $\SL_2(\Q)$ acts as isometries on $V$ by $gX :=g X g^{-1}$.

We let $D$ be the Grassmannian of lines in $V(\R) = V\otimes \R$ on which the quadratic form $Q$ is positive definite,
\[
D = \left\{z\subset V(\R);\ \text{dim}(z)=1 \text{ and } Q\vert_{z} >0 \right\},
\]
and we identify $D$ with the complex upper half-plane $\h$ by associating to $z = x + iy \in \h$ the positive line generated by
 \[
 X_{1}(z)=\frac{1}{\sqrt{2N}y} \begin{pmatrix} -x &|z|^{2} \\  -1 & x \end{pmatrix}.
 \]
The group $\SL_{2}(\R)$ acts on $\h$ by fractional linear transformations and
the identification above is $\SL_{2}(\R)$-equivariant, that is, $gX_{1}(z)=X_{1}(gz)$ for $g\in \SL_{2}(\R)$ and $z \in \h$.

Let $L\subset V$ be an even lattice of full rank. We write $L'$ for the dual lattice of $L$. Let $\G$ be a congruence subgroup of $\SL_{2}(\Q)$ that takes $L$ to itself and acts trivially on the discriminant group $L'/L$. Further, we let $M=\G\setminus D$ be the corresponding modular curve.

\begin{example}\label{ex:lattice}
	A particularly interesting lattice is given by
	\[
	L = \left\{ \begin{pmatrix}b & -a/N \\ c & -b\end{pmatrix}: a,b,c \in \Z\right\}.
	\]
	Its dual lattice is
	\[
	L' = \left\{ \begin{pmatrix}b/2N & -a/N \\ c & -b/2N\end{pmatrix}: a,b,c  \in\Z\right\}.
	\]
	We see that $L'/L$ is isomorphic to $\Z/2N\Z$ with quadratic form $x \mapsto -x^{2}/4N$. Thus the level of $L$ is $4N$. The group $\Gamma = \Gamma_{0}(N)$ acts on $L$ by conjugation and fixes the classes of $L'/L$, and the locally symmetric space is the modular curve $Y_{0}(N) = \Gamma\setminus D$ under the identification above. For fixed $m \in \Q$, the elements $X = \big(\begin{smallmatrix}b/2N & -a/N \\ c & -b/2N\end{smallmatrix}\big) \in L'$ with $Q(X) = m$ correspond to integral binary quadratic forms $[cN,-b,a] = cNx^{2} -bxy + ay^{2}$ of discriminant $-4Nm$, and this identification is compatible with the corresponding actions of $\Gamma_{0}(N)$.
\end{example}


\subsection{Cusps} 

We identify the set of isotropic lines $\mathrm{Iso}(V)$ in $V$
with $P^1(\Q)=\Q \cup \left\{ \infty\right\}$ via
\[
\psi: P^1(\Q) \rightarrow \mathrm{Iso}(V), \quad \psi((\alpha:\beta))
 = \mathrm{span}\left(\begin{pmatrix} \alpha\beta &\alpha^2 \\  -\beta^2 & -\alpha\beta \end{pmatrix}\right).
\]
The map $\psi$ is a bijection and $\psi(g(\alpha:\beta))=g\psi((\alpha:\beta))$ for $g \in \SL_{2}(\Q)$. 
Thus, the cusps of $M$ (i.e.~the $\G$-classes of $P^1(\Q)$) can be identified with the $\G$-classes of $\mathrm{Iso}(V)$.

If we set $\ell_\infty := \psi(\infty)$, then $\ell_\infty$ is spanned by 
$X_\infty=\left(\begin{smallmatrix}0 & 1 \\ 0 & 0\end{smallmatrix}\right)$. 
For $\ell \in \mathrm{Iso}(V)$ we pick $\sigma_{\ell} \in\SL_2(\Z)$ 
such that $\sigma_{\ell}\ell_\infty=\ell$. We let $\Gamma_{\ell}$ be the stabilizer of $\ell$ in $\Gamma$. Then $\sigma_{\ell}^{-1}\overline{\Gamma}_{\ell}\sigma_{\ell}$ is generated by $\left(\begin{smallmatrix}1 & \alpha_{\ell} \\ 0 & 1 \end{smallmatrix}\right)$ for some $\alpha_\ell \in \Q_{>0}$ which we call the width of the cusp $\ell$. For each $\ell$, there is a $\beta_{\ell} \in \Q_{>0}$ such that $\left(\begin{smallmatrix}0 & \beta_{\ell}  \\ 0 & 0\end{smallmatrix}\right)$ is a primitive element of $\ell_{\infty}\cap \sigma_{\ell}^{-1}L$. We write $\varepsilon_{\ell} = \alpha_{\ell}/\beta_{\ell}$. The quantities $\alpha_{\ell},\beta_{\ell}$ and $\varepsilon_{\ell}$ only depend on the $\Gamma$-class of $\ell$.

We compactify the modular curve $M$ to a compact Riemann surface $\overline{M}$ by adding a point for each cusp $\ell\in\G\setminus \Iso(V)$, and we denote this point again by $\ell$. 
Write $q_\ell=\exp(2\pi i \sigma_\ell^{-1}z/\alpha_\ell)$ for the chart around $\ell$. We define $D_{1/T}=\{w\in\C\,:\, |w|<\frac{1}{2\pi T}\}$ for $T>0$. Note that if $T$ is sufficiently big, then the inverse images $q_\ell^{-1}D_{1/T}$ are disjoint in $M$. We define the truncated modular curve by
\begin{align}\label{eq:TruncFun}
M_T= \overline{M}\setminus \coprod_{\ell\in\G\setminus \Iso(V)} q_\ell^{-1}D_{1/T}.
\end{align}


\subsection{The Weil representation and harmonic weak Maass forms}\label{sec:weil}
We recall the basic facts about vector valued harmonic weak Maass forms for the Weil representation from \cite{brfu04}.

By $\Mp_2(\Z)$ we denote the integral metaplectic group consisting of pairs
$(\gamma, \phi)$, where $\gamma = {\smallabcd \in \SL_2(\Z)}$ and $\phi:\h\rightarrow \C$
is a holomorphic function with $\phi(\tau)^{2}=c\tau+d$. We let $\C[L'/L]$ be the group ring of $L'/L$, generated by the basis vectors $\e_h$ for $h \in L'/L$ and equipped with the scalar product $\langle \e_{h} ,\e_{h'} \rangle = \delta_{h,h'}$, which is conjugate-linear in the second variable. The Weil representation $\rho_{L}$ of $\Mp_{2}(\Z)$ is a unitary representation which is defined for the generators $S=(\smallSmatrix,\sqrt{\tau})$ and $T=(\smallTmatrix, 1)$ of $\Mp_2(\Z)$ and $h \in \C[L'/L]$ by the formulas
\begin{align*}
 \rho_{L}(T) \e_h &= e(Q(h)) \e_h,\\
  \rho_{L}(S) \e_h &= \frac{\sqrt{i}}{\sqrt{\abs{L'/L}}}
				\sum_{h' \in L'/L} e(-(h',h)) \e_{h'},
\end{align*}
where $e(a):=e^{2\pi i a}$. The dual Weil representation will be denoted by $\overline{\rho}_{L}$.
 
A twice continuously differentiable function $F: \h \to \C[L'/L]$ is called a harmonic weak Maass form of weight $k \in \frac{1}{2}\Z$ with respect to $\rho_{L}$ if it satisfies
\begin{enumerate}
 \item $\Delta_k F=0$, where
 \[
 \Delta_k=-v^2\left(\frac{\partial^2}{\partial u^2}+\frac{\partial^2}{\partial v^2}\right)
+ikv\left(\frac{\partial}{\partial u}+i\frac{\partial}{\partial v}\right) \qquad 
\]
with $\tau = u + iv$ is the weight $k$ hyperbolic Laplace operator.
 \item $F(\gamma \tau) = \phi(\tau)^{2k} \rho_{L}(\gamma, \phi)F(\tau)$ for all $(\gamma,\phi) \in \Mp_2(\Z)$,
\item There is a Fourier polynomial $P_{F}(\tau) = \sum_{h \in L'/L}\sum_{n \leq 0}a^{+}(n,h)q^{n}\e_{h}$, called the principal part of $F$, such that 
\[
F(\tau) - P_{F}(\tau)  =O(e^{-\varepsilon v})
\]
as $v \to \infty$, uniformly in $u$, for some $\varepsilon > 0$.
 \end{enumerate}
We denote the space of such functions by $H^{+}_{k,\rho_{L}}$. Further, we let $M^{\text{!}}_{k,\rho_{L}}$ be the subspace of weakly holomorphic modular forms (consisting of the holomorphic forms in $H^{+}_{k,\rho_{L}}$).

A harmonic weak Maass form uniquely decomposes into a holomorphic and a non-holomorphic part $F=F^{+}+F^{-}$ with Fourier expansions
\begin{align*}
F^{+}(\tau)&=\sum\limits_{h\in 
L'/L}\sum\limits_{n\gg -\infty}a^{+}(n,h)q^{n} \mathfrak{e}_h,
\\
F^{-}(\tau)&=\sum\limits_{h\in L'/L} \sum\limits_{n < 0}a^{-}(n,h)\G(1-k,4\pi |m|v)q^{n}\mathfrak{e}_h,
\end{align*}
where $\G(s,x) = \int_{x}^{\infty}t^{s-1}e^{-t}dt$ denotes the incomplete $\G$-function.

The space $H_{k}^{+}(\Gamma)$ of scalar valued harmonic weak Maass forms of weight $k \in \Z$ for $\Gamma$ is defined analogously, with a growth condition as in (3) at each cusp.

\subsection{Differential Operators}\label{sec:diffop}
The Maass lowering and raising operators are defined by
\[
L_k=-2iv^2\frac{\partial}{\partial \bar{\tau}} \quad\quad\quad\mathrm{and}
\quad\quad\quad R_k=2i\frac{\partial}{\partial\tau}+kv^{-1}.
\]
They lower or raise the weight of an automorphic form of weight $k$ by $2$.
 Moreover, these operators commute with the slash operator
and they are related to the weighted Laplace operator by
\begin{align}\label{eq:DkLkRk}
-\Delta_{k} &= L_{k+2}R_{k}+k = R_{k-2}L_{k}.
\end{align}
This implies the commutation relations
\begin{align}
\label{eq:RkDk}
 R_k\Delta_k & =(\Delta_{k+2}-k)R_k, \\
\label{eq:LkDk}
 \Delta_{k-2}L_k &= L_k(\Delta_k+2-k).
\end{align}
We also define iterated versions of the lowering and raising operators by
\[
 L_k^n=L_{k-2(n-1)}\circ\cdots L_{k-2}\circ L_{k},
 \quad\quad
 R_k^n=R_{k+2(n-1)}\circ \cdots\circ R_{k+2} \circ R_{k}.
\]
For $n=0$ we set $L_k^0=R_k^0=\text{id}$. Using (\ref{eq:RkDk}) and (\ref{eq:LkDk}) inductively one can find many interesting commutation relations between the iterated lowering and raising operators and the weighted Laplacian. We collect some identities for later use.
\begin{lemma}\label{lm:reldiff}
For $k \in \Z_{\geq 0}$ and $\ell = 0,\dots,k$ we have
			\[
			\Delta_{-2\ell}R_{-2k}^{k-\ell} = R_{-2k}^{k-\ell}\left(\Delta_{-2\ell}-(k-\ell)(k+\ell+1)\right).
			\]
If $k$ is even then
\begin{align*}
\Delta_{1/2-k}L_{1/2}^{k/2}&=L_{1/2}^{k/2} \left(\Delta_{1/2}+\frac{k}{4}(k+1)\right), \\
 \Delta_{3/2+k}R_{3/2}^{k/2}&=R_{3/2}^{k/2}\left(\Delta_{3/2}+\frac{k}{4}(k+1)\right),
\end{align*}
and if $k$ is odd we have
\begin{align*}
 \Delta_{1/2-k}L_{3/2}^{(k+1)/2}&=L_{3/2}^{(k+1)/2}\left(\Delta_{3/2}+\frac{k}{4}(k+1)\right),
 \\
\Delta_{3/2+k}R_{1/2}^{(k+1)/2}&=R_{1/2}^{(k+1)/2} \left(\Delta_{1/2}+\frac{k}{4}(k+1)\right).
\end{align*}
\end{lemma}

We also require the antilinear differential operator
\[
\xi_k F=v^{k-2}\overline{L_k F(\tau)}=R_{-k}v^k\overline{F(\tau)} = 2iv^{k}\overline{\frac{\partial}{\partial \bar{\tau}}F(\tau)}
\] 
from \cite{brfu04}. It defines a surjective map $\xi_k: H^+_{k,\rho_{L}}\rightarrow S_{2-k,\bar{\rho}_{L}}$ and acts on the Fourier expansion of $F \in H_{k,\rho_{L}}^{+}$ by
\[
\xi_{k}F = \xi_{k}F^{-} = -\sum_{h \in L'/L}\sum_{n > 0}\overline{a^{-}(-n,h)}(4\pi n)^{1-k}q^{n}\e_{h}.
\]

	\section{Theta Functions}\label{sec:thetafunctions}
	
In this section we introduce the theta functions that we will employ as kernel functions for the lifts we investigate in this paper. As before we let $L \subseteq V$ be an even lattice with dual lattice $L'$ and we let $\Gamma$ be a congruence subgroup of $\SL_{2}(\Q)$ which maps $L$ to itself acts trivially on $L'/L$.

For $z = x+iy \in \mathbb{H}$ the vectors
\begin{align*}
	X_{1}(z) &= \frac{1}{\sqrt{2N}y}\begin{pmatrix}-x & x^{2} + y^{2} \\ -1 & x \end{pmatrix},\\
	X_{2}(z) &= \frac{1}{\sqrt{2N}y}\begin{pmatrix}x & -x^{2} + y^{2} \\ 1 & -x \end{pmatrix}, \\
	X_{3}(z) &= \frac{1}{\sqrt{2N}y}\begin{pmatrix}y & -2xy \\ 0 & -y \end{pmatrix},
\end{align*}
form an orthogonal basis of $V(\R)$ with \[(
X_{1}(z),X_{1}(z)) = 1\quad\text{and} \quad(X_{2}(z),X_{2}(z)) = (X_{3}(z),X_{3}(z)) = -1.
\] For $z \in \h$ and $X = \left( \begin{smallmatrix}x_{2} & x_{1} \\ x_{3} & -x_{2}\end{smallmatrix}\right) \in V(\R)$ we define the quantities
\begin{align*}
p_{z}(X) &= \sqrt{2}(X,X_{1}(z)) = -\frac{\sqrt{N}}{y}(x_{3}|z|^{2}-2x_{2}x -x_{1}), \\
Q_{X}(z) &= \sqrt{2N}y(X,X_{2}(z)+iX_{3}(z)) = N(x_{3}z^{2}-2x_{2}z-x_{1}).
\end{align*}
Further, we let
\[
 R(X,z)= \frac{1}{2}p^{2}_{z}(X)-(X,X).
\]
Note that $R(X,z)$ is non-negative and equals $0$ if and only if $X \in \R X_{1}(z)$. For $\gamma \in \SL_{2}(\R)$ we have
\begin{align}\label{zTransformationRules}
p_{\gamma z}(X) = p_{z}(\gamma^{-1}X), \quad Q_{X}(\gamma z) = j(\gamma,z)^{-2}Q_{\gamma^{-1}X}(z), \quad R(X,\gamma z) = R(\gamma^{-1}X,z),
\end{align}
which can be verified by a direct calculation.

\begin{example}
For the lattice $L$ as in Example~\ref{ex:lattice} and $X = \big(\begin{smallmatrix}b/2N & -a/N \\ c & -b/2N \end{smallmatrix} \big) \in L'$ we have $Q_{X}(z) = cNz^{2}-bz+a$ and $-y\sqrt{N}p_{z}(X) = cN|z|^{2}-bx+a$, i.e.~these quantities are related to CM-points and geodesics associated to the quadratic form $[cN,-b,a]$.
\end{example}

For $\tau = u + iv,z = x+iy \in \mathbb{H}$ and $k \in \Z_{\geq 0}$ we define Schwartz functions on $V(\R)$ by
\begin{align*}
\psi_{M,k}(X,\tau,z) &= v^{k+1}p_{z}(X)Q^{k}_{X}(\bar{z})e^{-2\pi vR(X,z)}e^{2\pi i Q(X)\tau}, \\
\varphi_{KM}(X,\tau,z) &= \left(vp_{z}^{2}(X) - \frac{1}{2\pi} \right)e^{-2\pi vR(X,z)}e^{2\pi i Q(X)\tau},\\
\varphi_{Sh,k}(X,\tau,z) &= v^{1/2}y^{-2k-2}Q^{k+1}_{X}(\bar{z})e^{-2\pi v R(X,z)}e^{2\pi i Q(X)\tau},
\end{align*}
which we call the \emph{Millson}, the \emph{Kudla-Millson} and the \emph{Shintani Schwartz function}, respectively. They have been studied in many recent works, for example \cite{brfu04,brfu06, hoevel,bif,crawford}.
To each Schwartz function $\varphi(X,\tau,z)$ on $V(\R)$ of this shape we associate a theta function
\[
\Theta(\tau,z,\varphi) =  \sum_{h \in L'/L}\sum_{X \in h+L}\varphi(X,\tau,z)\e_{h},
\] 
which defines a smooth $\C[L'/L]$-valued function in $\tau$ and $z$. 
We summarize the transformation properties of the theta functions associated to our special Schwartz functions.

\begin{proposition}\label{prop:propertiestheta}
	Let $k \in \Z$, $k \geq 0$.
	\begin{enumerate}
		\item The Millson theta function $\Theta(\tau,z,\psi_{M,k})$ has weight $1/2-k$ in $\tau$ for the representation $\rho_{L}$ and $\overline{\Theta(\tau,z,\psi_{M,k})}$ has weight $-2k$ in $z$ for $\Gamma$.
		\item The Kudla-Millson theta function $\Theta(\tau,z,\varphi_{KM})$ has weight $3/2$ in $\tau$ for the representation $\rho_{L}$ and is $\Gamma$-invariant in $z$.
		\item The Shintani theta function $\overline{\Theta(\tau,z,\varphi_{Sh,k})}$ has weight $k+3/2$ in $\tau$ for the representation $\overline{\rho}_{L}$ and $\Theta(\tau,z,\varphi_{Sh,k})$ has weight $2k+2$ in $z$ for $\Gamma$.
	\end{enumerate}
\end{proposition}

\begin{proof}
	The transformation behaviour of the three theta functions is certainly well known and can be found in the literature above. For convenience, we sketch the proof:

	The behaviour in $z$ easily follows from the rules (\ref{zTransformationRules}), and the behaviour in $\tau$ can be determined using general results from \cite{borcherds}: Let $\R^{1,2}=(\R^{3},(x,x) = x_{1}^{2}-x_{2}^{2}-x_{3}^{2})$ be the standard quadratic space of signature $(1,2)$. Under the isometry $V(\R) \cong \R^{1,2}$ given by $\sum_{i=1}^{3}\alpha_{i}X_{i}(z) \mapsto (\alpha_{1},\alpha_{2},\alpha_{3})$ the functions $p_{z}(X)$ and $Q_{X}(\bar{z})$ correspond to the polynomials
	\[
	p_{z}(x_{1},x_{2},x_{3}) = \sqrt{2}x_{1}, \qquad Q_{(x_{1},x_{2},x_{3})}(\bar{z}) = -\sqrt{2N}y(x_{2}-ix_{3}),
	\]
	on $\R^{1,2}$. They are homogeneous of degree $(1,0)$ and $(0,1)$, respectively.
	Further, the Millson, Kudla-Millson and Shintani theta functions are (up to some powers of $v$) the theta functions associated to the lattice $L$ and the polynomials $p_{z}(X)Q_{X}^{k}(\bar{z}), p_{z}^{2}(X)$ and $y^{-2k-2}Q_{X}^{k+1}(\bar{z})$ as in \cite{borcherds}, Section 4. The transformation behaviour in $\tau$ now follows from Theorem 4.1. in \cite{borcherds}. 
\end{proof}

We want to investigate the growth of the theta functions at the cusps of $\Gamma$. To describe this in a convenient way, we follow the ideas of \cite[Section 2.2]{bif} and define certain theta functions associated to the cusps. 

For an isotropic line $\ell \in \Iso(V)$ the space $W_{\ell} = \ell^{\perp}/\ell$ is a unary negative definite quadratic space with the quadratic form $Q(X + \ell) := Q(X)$, and
\[
K_{\ell} = (L\cap \ell^{\perp})/(L \cap \ell)
\]
is an even lattice with dual lattice
\[
K'_{\ell} = (L' \cap \ell^{\perp})/(L' \cap \ell).
\]
The vector $X_{\ell}=\sigma_{\ell}.X_{3}(i)$ is a basis of $W_{\ell}$ with $(X_{\ell},X_{\ell}) = -1$, and for $k \in \Z_{\geq 0}$ the polynomial $p_{\ell,k}(X) = (-\sqrt{2N}i(X,X_{\ell}))^{k}$ is homogeneous of degree $(0,k)$. We let $\Theta_{\ell,k}(\tau)$ be the theta function associated to $K_{\ell}$ and $p_{\ell,k}$ as in \cite{borcherds}, Section 4. By \cite[Theorem 4.1.]{borcherds} the complex conjugate $\overline{\Theta_{\ell,k}(\tau)}$ is a holomorphic modular form of weight $k+1/2$ for the dual Weil representation of $K_{\ell}$. Using \cite[Lemma 5.6.]{brhabil}, it gives rise to a holomorphic modular form of weight $k+1/2$ for the dual Weil representation $\overline{\rho}_{L}$ of $L$, which we also denote by $\overline{\Theta_{\ell,k}(\tau)}$. It is a cusp form if $k > 0$.


%

\begin{proposition}\label{prop:growththeta}
  Let $\ell$ be a cusp of $\Gamma$.
  \begin{enumerate}
  \item For the Millson theta function we have
 \begin{align*}
 \Theta(\tau,\sigma_\ell z,\psi_{M,k})&= O(e^{-Cy^2}),
  \end{align*}
  if $k = 0$, and
   \begin{align*}
 j(\sigma_{\ell},\bar{z})^{2k}\Theta(\tau,\sigma_\ell z,\psi_{M,k})&=
  -y^{k+1}\frac{k}{2\pi\beta_{\ell}}v^{k-1/2}\Theta_{\ell,k-1}(\tau) + O(e^{-Cy^{2}}),
  \end{align*}
 if $k > 0$, as $y \rightarrow\infty$, uniformly in $x$, for some constant $C>0$.
  \item For the Kudla-Millson theta function we have
  \begin{align*}
  \Theta(\tau,\sigma_{\ell}z,\varphi_{KM}) &= O(e^{-Cy^{2}}),
  \end{align*}
  as $y \rightarrow\infty$, uniformly in $x$, for some constant $C>0$.
  \item For the Shintani theta function we have
  \begin{align*}
  j(\sigma_{\ell},z)^{-2k-2}\Theta(\tau,\sigma_{\ell}z,\varphi_{Sh,k}) &= y^{-k}\frac{1}{\sqrt{N}\beta_{\ell}}\Theta_{\ell,k+1}(\tau) + O(e^{-Cy^{2}}),
 \end{align*}
 as $y \rightarrow\infty$, uniformly in $x$, for some constant $C>0$.
 \end{enumerate}
 Moreover, all of the partial derivates of the functions hidden in the $O$-notation are square exponentially decreasing as $y \to \infty$.
\end{proposition}

\begin{proof}
	Using the rules (\ref{zTransformationRules}) we can write
	\[
	j(\sigma_{\ell},\bar{z})^{2k}\Theta(\tau,\sigma_\ell z,\psi_{M,k}) = \sum_{h \in (\sigma_{\ell}^{-1}L)'/(\sigma_{\ell}^{-1}L)}\sum_{X \in h+(\sigma_{\ell}^{-1}L)}\psi_{M,k}(X,\tau,z)\e_{h},
	\]
	and similarly for the other two theta functions, so we can equivalently estimate the growth of the theta functions for the lattice $\sigma_{\ell}^{-1}L$ at the cusp $\infty$. The result now follows from Theorem 5.2 in \cite{borcherds} applied to the lattice $\sigma_{\ell}^{-1}L$ and the primitive isotropic vector $\left(\begin{smallmatrix}0 & \beta_{\ell} \\ 0 & 0\end{smallmatrix}\right) \in  \ell_{\infty}\cap\sigma_{\ell}^{-1}L$.
\end{proof}



%

%
The theta functions we just defined satisfy some interesting differential equations. All of the following identities can be checked on the level of Schwartz functions by a direct computation using the rules
\begin{align}\label{eq:diffpQR}
&\frac{\partial}{\partial z} y^{-2}Q_{X}(z)=-i\sqrt{N}y^{-2}p_z(X),  &\frac{\partial}{\partial z} p_{z}(X) =-\frac{i}{2\sqrt{N}}y^{-2}Q_{X}(\bar{z}), \\
&\notag \frac{\partial}{\partial z} R(X,z) =-\frac{i}{2\sqrt{N}}y^{-2}p_{z}(X)Q_{X}(\bar{z}), &y^{-2}Q_{X}(z)Q_{X}(\bar{z}) = 2NR(X,z).
\end{align}

\begin{lemma}\label{lm:iddelta}
	For $k \geq 0,$ we have
\begin{align*}
 \Delta_{1/2-k,\tau}\Theta(\tau,z,\psi_{M,k}) &= \frac{1}{4} \overline{\Delta_{-2k,z}\overline{\Theta(\tau,z,\psi_{M,k})}},\\
 \Delta_{3/2,\tau}\Theta(\tau,z,\varphi_{KM}) &= \frac{1}{4} \Delta_{0,z}\Theta(\tau,z,\varphi_{KM}), \\
 \overline{\Delta_{k+3/2,\tau}\overline{\Theta(\tau,z,\varphi_{Sh,k})}} &= \frac{1}{4}\Delta_{2k+2,z}\Theta(\tau,z,\varphi_{Sh,k}).
\end{align*}
\end{lemma}

\begin{proof}
	Compare \cite[Proposition 3.10]{hoevel}, \cite[Proposition 4.5]{brhabil}.
\end{proof}

The Millson and the Shintani theta function are related by the following identity.

\begin{lemma}\label{lm:relshinmillson}
	For $k \geq 0$ we have
\begin{align*}
\xi_{1/2-k,\tau} \Theta(\tau,z,\psi_{M,k}) = \frac{1}{2\sqrt{N}}\xi_{2k+2,z}\Theta(\tau,z,\varphi_{Sh,k}).
\end{align*}
\end{lemma}

\begin{proof}
	Compare \cite[Lemma 3.3]{brikavia} or \cite[Lemma 7.2.1]{crawford}.
\end{proof}

We will also need the following relation between Millson theta functions of different weights and the Millson and Kudla-Millson theta functions:

\begin{lemma}\label{lm:relmillsondiffweight}
	For $k \geq 0$ we have
		\begin{align*}
		L_{-2k-2,z}L_{-2k,z}\overline{L_{1/2-k,\tau}\Theta(\tau,z,\psi_{M,k})} = \frac{\pi}{N}\left(\Delta_{-2k-4,z}-4k-6\right)\overline{\Theta(\tau,z,\psi_{M,k+2})}.
		\end{align*}
	Further, we have
	\begin{align*}
	L_{0,z}\overline{L_{3/2,\tau}\Theta(\tau,z,\varphi_{KM})} = -\frac{1}{2\sqrt{N}}(\Delta_{-2,z}-2)\overline{\Theta(\tau,z,\psi_{M,1})}.
	\end{align*}
\end{lemma}

\begin{proof}
	This can be shown by a direct calculation using the rules (\ref{eq:diffpQR}).
\end{proof}

	\section{Theta Lifts}\label{sec:thetalifts}

	Let $k \in \Z_{\geq 0}$ and let $F \in H_{-2k}^{+}(\Gamma)$ be a harmonic weak Maass form. We would like to integrate $F$ against the Millson theta function $\Theta(\tau,z,\psi_{M,k})$ on $M = \Gamma \setminus \h$ to obtain a function that transforms like a modular form of weight $1/2-k$. Unfortunately, Proposition~\ref{prop:growththeta} shows that the integral does not converge for $k > 0$, so it has to be regularized in a suitable way. Using the regularization of \cite{borcherds}, we define the Millson theta lift by
	\[
	\IM(F,\tau) = \lim_{T\to \infty}\int_{M_{T}}F(z)\Theta(\tau,z,\psi_{M,k})y^{-2k}d\mu(z).
	\]
	Note that we integrate in the orthogonal variable $z$ here. The integral in the symplectic variable $\tau$ was considered previously in \cite{hoevel}, \cite{brikavia} and \cite{crawford}. It was shown that the corresponding lift has jump singularities along certain geodesics in the upper-half plane, which led to the discovery of locally harmonic Maass forms. Similarly, the theta lifts investigated in the fundamental works of Borcherds \cite{borcherds} and Bruinier \cite{brhabil} (which are integrals in the $\tau$-variable) have singularities along Heegner divisors in $\h$. In contrast to these singular lifts, it turns out that the Millson theta lift is in fact harmonic on the upper half-plane.
	
	
	\begin{proposition}\label{prop:convergencethetalift}
			For $k \in \Z_{\geq 0}$ the Millson theta lift $\IM(F,\tau)$ of $F \in H_{-2k}^{+}(\Gamma)$ is a harmonic function that transforms like a modular form of weight $1/2-k$ for $\rho_{L}$.
	\end{proposition}
	
	\begin{proof}
		Unwinding the definition of the truncated surface $M_{T}$, we see that it suffices to show that the limit
		\begin{align*}
		\lim_{T \to \infty}\int_{1}^{T}\int_{0}^{\alpha_{\ell}}F_{\ell}(z)j(\sigma_{\ell},\bar{z})^{2k}\Theta(\tau,\sigma_{\ell}z,\psi_{M,k})y^{-2k-2}dx dy
		\end{align*}
		exists for every cusp $\ell \in \Iso(V)$, where $F_{\ell} = F|_{-2k}\sigma_{\ell}$.  For $k = 0$ Proposition~\ref{prop:growththeta} states that the Millson theta function is square exponentially decreasing at all cusps, so the integral actually converges without regularization in this case. For $k > 0$ we see by the same lemma that it suffices to show that
		\[
		\lim_{T \to \infty}\int_{1}^{T}\int_{0}^{\alpha_{\ell}}F_{\ell}(z)y^{-k-1}dx dy
		\]
		exists. But the integral over $x$ picks out the constant coefficient $a_{\ell}^{+}(0)$ of $F_{\ell}$, and the limit of the remaining integral over $y$ gives $\frac{1}{k}$. This shows that $\IM(F,\tau)$ is well-defined. The transformation behaviour of the Millson theta function implies that $\IM(F,\tau)$ has weight $1/2-k$ for $\rho_{L}$.

		To prove that $\IM(F,\tau)$ is harmonic we first use Lemma~\ref{lm:iddelta} to write
		\[
		\Delta_{1/2-k,\tau}\IM(F,\tau) = \lim_{T \to \infty}\frac{1}{4}\int_{M_{T}}F(z)\overline{\Delta_{-2k,z}\overline{\Theta(\tau,z,\psi_{M,k})}}y^{-2k}d\mu(z).
		\]
		By Lemma 4.3. in \cite{brhabil}, or more generally Stokes' Theorem, we have
		\begin{align*}
		&\int_{M_{T}}F(z)\overline{\Delta_{-2k,z}\overline{\Theta(\tau,z,\psi_{M,k})}}y^{-2k}d\mu(z) - \int_{M_{T}}\Delta_{-2k,z}F(z)\Theta(\tau,z,\psi_{M,k})y^{-2k}d\mu(z) \\
		&= \int_{\partial M_{T}} L_{-2k,z}F(z)\Theta(\tau,z,\psi_{M,k})y^{-2k-2}dz - \int_{\partial M_{T}} F(z)\overline{L_{-2k,z}\overline{\Theta(\tau,z,\psi_{M,k})}}y^{-2k-2}dz.
		\end{align*}  
		Now it follows easily from the growth estimates in Proposition~\ref{prop:growththeta} that the boundary integrals vanish in the limit. Since $F$ is harmonic, we obtain $\Delta_{1/2-k,\tau}\IM(F,\tau) = 0$.
	\end{proof}
	
	We define the Shintani theta lift of a cusp form $G \in S_{2k+2}(\Gamma)$ by
	\begin{align*}
	\ISh(G,\tau) = \int_{M}G(z)\overline{\Theta(\tau,z,\varphi_{Sh,k})}y^{2k+2}d\mu(z).
	\end{align*}
	The rapid decay of $G$ at the cusps and similar arguments as above show that $\ISh(G,\tau)$ converges to a harmonic function which transforms like a modular form of weight $3/2+k$ for $\overline{\rho}_{L}$. The Millson and the Shintani theta lifts are related by the following identity.
	
	\begin{proposition}\label{prop:xidiagram}
		For $F \in H_{0}^{+}(\Gamma)$ we have
		\[
		\xi_{1/2,\tau}(\IM(F,\tau)) = -\frac{1}{2\sqrt{N}}\ISh(\xi_{0,z}F,\tau)+\frac{1}{2N}\sum_{\ell \in \Gamma \setminus \Iso(V)}\varepsilon_{\ell}\overline{a_{\ell}^{+}(0)\Theta_{\ell,1}(\tau)},
		\]
		and for $k \in \Z_{> 0}$ and $F \in H_{-2k}^{+}(\Gamma)$ we have
		\[
		\xi_{1/2-k,\tau}(\IM(F,\tau)) = -\frac{1}{2\sqrt{N}}\ISh(\xi_{-2k,z}F,\tau).
		\]
	\end{proposition}

	\begin{proof}
		By Lemma~\ref{lm:relshinmillson} we have for $k \in \Z_{\geq 0}$
		\begin{align*}
		\xi_{1/2-k,\tau}(\IM(F,\tau)) = \lim_{T \to \infty}\frac{1}{2\sqrt{N}}\int_{M_{T}}\overline{F(z)}\xi_{2k+2,z}\Theta(\tau,z,\varphi_{Sh,k})y^{-2k}d\mu(z).
		\end{align*}
		Using Stokes' Theorem we obtain
		\begin{align*}
		\int_{M_{T}}\overline{F(z)}\xi_{2k+2,z}\Theta(\tau,z,\varphi_{Sh,k})y^{-2k}d\mu(z) &= -\int_{M_{T}}\xi_{-2k,z}F(z)\overline{\Theta(\tau,z,\varphi_{Sh,k})}y^{2k+2}d\mu(z) \\
		&\quad-\int_{\partial M_{T}}\overline{F(z)\Theta(\tau,z,\varphi_{Sh,k})} d\bar{z}.
		\end{align*}
		The limit of the first integral on the right-hand side is $\ISh(\xi_{-2k,z}F,\tau)$. The boundary integral can be written as
		\[
		-\int_{\partial M_{T}}\overline{F(z)\Theta(\tau,z,\varphi_{Sh,k})} d\bar{z} = \sum_{\ell \in \Gamma \setminus \Iso(V)}\int_{iT}^{\alpha_{\ell}+iT}\overline{F_{\ell}(z)j(\sigma_{\ell},z)^{-2k-2}\Theta(\tau,\sigma_{\ell}z,\varphi_{Sh,k})}d\bar{z},
		\]
		where $F_{\ell} = F|_{-2k}\sigma_{\ell}$. Using Proposition~\ref{prop:growththeta} and carrying out the integral we see that the right-hand side vanishes in the limit if $k > 0$ and equals
		\[
		\frac{1}{\sqrt{N}}\sum_{\ell \in \Gamma \setminus \Iso(V)}\varepsilon_{\ell}\overline{a_{\ell}^{+}(0)\Theta_{\ell,1}(\tau)}
		\]
		if $k = 0$. This completes the proof.
	\end{proof}
	
	We summarize the most important mapping properties of the Millson and the Shintani theta lift in the following theorem.
	
	\begin{theorem}\label{thm:PropertiesLifts}~
		\begin{enumerate}
			\item The Millson theta lift maps $H^{+}_{-2k}(\Gamma)$ to $H^{+}_{1/2-k,\rho_{L}}$ for $k \geq 0$.

			\item The Millson theta lift maps $M_{0}^{!}(\Gamma)$ to $H^{+}_{1/2,\rho_{L}}$ and $M_{-2k}^{!}(\Gamma)$ to $M_{1/2-k,\rho_{L}}^{!}$ for $k > 0$.
						\item The Shintani theta lift maps $S_{2k+2}(\Gamma)$ to $S_{3/2+k,\overline{\rho}_{L}}$ for $k \geq 0$.
		\end{enumerate}
	\end{theorem}
	
	\begin{proof}
		For the first item it remains to compute the Fourier expansion of the Millson theta lift, which will be done in Section~\ref{sec:fourexp}. The second claim then follows immediatly from Proposition~\ref{prop:xidiagram} if we use that $\xi_{-2k}$ annihilates holomorphic functions and that $\overline{\Theta_{\ell,1}(\tau)}$ is a cusp form of weight $3/2$ for $\overline{\rho}_{L}$. The third item then follows by combining the first item with Proposition~\ref{prop:xidiagram} and the fact that $\xi_{-2k}:H^{+}_{-2k}(\Gamma) \to S_{2k+2}(\Gamma)$ is surjective, see \cite[Theorem 3.7.]{brfu04}.
	\end{proof}
	
	The square exponential decay of the $k = 0$ Millson theta function and the Kudla-Millson theta function at the cusps implies that the integral of a harmonic weak Maass form $F \in H_{0}^{+}(\Gamma)$ against each of the two theta functions over $M$ converges without regularization. Following an idea of \cite{brono2}, we  define theta lifts of $F \in H_{-2k}^{+}(\Gamma)$ by first raising it to a $\Gamma$-invariant function, integrating it against the two theta functions, and then applying suitable differential operators to make the result harmonic again. To make this precise let $k \in \Z_{\geq 0}$ and $F \in H_{-2k}^{+}(\Gamma)$. We define
	\begin{align*}
		\LambdaM (F,\tau)&= \begin{dcases}  L_{1/2,\tau}^{k/2}\int_{M}R_{-2k,z}^{k}F(z)\Theta(\tau,z,\psi_{M})d\mu(z), & \text{if $k$ is even,}\\
		L_{3/2,\tau}^{(k+1)/2}\int_{M}R_{-2k,z}^{k}F(z)\Theta(\tau,z,\varphi_{KM})d\mu(z), &\text{if $k$ is odd.} \end{dcases}
	\end{align*}

	\begin{proposition}
	Let $k \in \Z_{\geq 0}$ and $F \in H_{-2k}^{+}(\Gamma)$.
	The theta lift $\LambdaM(F,\tau)$ is a harmonic function which transforms like a modular form of weight $1/2-k$ for $\rho_{L}$.
	\end{proposition}
	
	\begin{proof}
		By what we have said above, all integrals converge. The transformation behaviour is then obvious. To prove that the lifts are harmonic we use the relations in Lemma~\ref{lm:reldiff}, Lemma~\ref{lm:iddelta} and Stokes' Theorem as above. We leave the details to the reader.
	\end{proof}

	\begin{remark}
	Similarly, we can define a theta lift
		\begin{align*}
		\widetilde\Lambda^{\mathrm{M}} (F,\tau)&= \begin{dcases}  R_{1/2,\tau}^{(k+1)/2}\int_{M}R_{-2k,z}^{k}F(z)\Theta(\tau,z,\psi_{M})d\mu(z), & \text{if $k$ is odd,}\\
		R_{3/2,\tau}^{k/2}\int_{M}R_{-2k,z}^{k}F(z)\Theta(\tau,z,\varphi_{KM})d\mu(z), &\text{if $k$ is even.} \end{dcases}
	\end{align*}
	This gives a weakly holomorphic modular form of weight $3/2+k$ for $\rho_{L}$ if $k > 0$ (see \cite{alfes,alfesdiss}). The case $k=0$ was considered by Bruinier and Funke in \cite{brfu06}.
	\end{remark}
	
	We now want to show that the regularized Millson theta lift $\IM(F,\tau)$ defined above agrees with $\LambdaM(F,\tau)$ up to some constant. This will be useful when we compute the Fourier coefficients of $\IM(F,\tau)$.
	
	\begin{theorem}\label{thm:relationlifts}
		Let $k \in \Z_{\geq 0}$ and $F \in H_{-2k}^{+}(\Gamma)$. Then
		\begin{align*}
		\IM(F,\tau) = \bigg(\left(-\frac{\pi}{N}\right)^{k/2}\prod_{j = 0}^{k/2-1}(k-2j)(k+2j+1)\bigg)^{-1} \LambdaM(F,\tau),
		\end{align*}
		if $k$ is even, and
		\begin{align*}
		\IM(F,\tau) = \bigg(-\frac{1}{2\sqrt{N}}\left( -\frac{\pi}{N}\right)^{(k-1)/2}\prod_{j=0}^{(k-1)/2}(k-2j+1)(k+2j)\bigg)^{-1}\LambdaM(F,\tau),
		\end{align*}
		if $k$ is odd.
	\end{theorem}
	
	\begin{proof}
		The proof involves several applications of Stokes' Theorem. Using Proposition~\ref{prop:growththeta} it is straightforward but tedious to verify that all boundary integrals vanish in the limit. We leave these verifications to the suspicious reader and omit all boundary integrals to simplify the exposition.
		
		Let $k$ be even. We consider the expression
		\begin{align*}
		I_{j}(F,\tau)=\lim_{T \to \infty }L_{1/2-2j,\tau}^{k/2-j}\int_{M_{T}}R_{-2k,z}^{k-2j}F(z)\Theta(\tau,z,\psi_{M,2j})y^{-4j}d\mu(z)
		\end{align*}
		for $0 \leq j \leq k/2$. By the same arguments as above, it converges to a harmonic function of weight $1/2-k$ for $\rho_{L}$ which equals $\LambdaM(F,\tau)$ for $j = 0$ and $\IM(F,\tau)$ for $j = k/2$. We split off the innermost lowering operator in $\tau$ and the two outermost raising operators in $z$ and apply \cite[Lemma 4.2]{brhabil} (an instance of Stokes' Theorem) twice to see that $I_{j}(F,\tau)$ equals
		\begin{align*}
		\lim_{T \to \infty }L_{1/2-2j-2,\tau}^{k/2-j-1} \int_{M_{T}}R_{-2k,z}^{k-2j-2}F(z)\overline{L_{-4j-2,z}L_{-4j,z}\overline{L_{1/2-2j,\tau}\Theta(\tau,z,\psi_{M,2j})}}y^{-4j-4}d\mu(z).
		\end{align*}
		By Lemma~\ref{lm:relmillsondiffweight} we have
		\begin{align*}
		L_{-4j-2,z}L_{-4j,z}\overline{L_{1/2-2j,\tau}\Theta(\tau,z,\psi_{M,2j})} = \frac{\pi}{N}(\Delta_{-4j-4,z}-8j-6)\overline{\Theta(\tau,z,\psi_{M,2j+2})}.
		\end{align*}
		Using \cite[Lemma 4.3.]{brhabil} we now move the Laplace operator to $R_{-2k,z}^{k-2k-2}F$ in the integral over $M_{T}$. Lemma~\ref{lm:reldiff} shows that 
		\[
		\Delta_{-4j-4,z}R_{-2k}^{k-2j-2}F = -(k-2j-2)(k+2j+3) R_{-2k,z}^{k-2j-2}F,
		\]
		so together we obtain after a short calculation
		\[
		I_{j}(F,\tau) = -\frac{\pi}{N}(k-2j)(k+2j+1)I_{j+1}(F,\tau).
		\]
		The formula for even $k$ now follows inductively.
		
		For odd $k$ we first split off the innermost lowering operator in $\tau$ and the outermost raising operator in $z$ in $\LambdaM(F,\tau)$ and apply \cite[Lemma 4.2.]{brhabil} to get
		\begin{align*}
		\LambdaM(F,\tau) = -\lim_{T \to \infty}L_{-1/2,\tau}^{(k-1)/2}\int_{M_{T}}R_{-2k,z}^{k-1}F(z)\overline{L_{0,z}\overline{L_{3/2,\tau}\Theta(\tau,z,\varphi_{KM})}}y^{-2}d\mu(z).
		\end{align*}
		By Lemma~\ref{lm:relmillsondiffweight} we have
		\[
		L_{0,z}\overline{L_{3/2,\tau}\Theta(\tau,z,\varphi_{KM})} = -\frac{1}{2\sqrt{N}}(\Delta_{-2,z} - 2)\overline{\Theta(\tau,z,\psi_{M,1})}.
		\]
		Moving the Laplace operator to $R_{-2k}^{k-1}F$ (see \cite[Lemma 4.3.]{brhabil}) and using that $\Delta_{-2,z}R_{-2k,z}^{k-1}F = -(k-1)(k+2)R_{-2k,z}^{k-1}F$ we arrive at
		\[
		\LambdaM(F,\tau) = -\frac{1}{2\sqrt{N}}k(k+1)\lim_{T \to \infty}L_{-1/2,\tau}^{(k-1)/2}\int_{M_{T}}R_{-2k,z}^{k-1}F(z)\Theta(\tau,z,\psi_{M,1})y^{-2}d\mu(z).
		\]
		Similarly as in the even $k$ case we consider
		\begin{align*}
		I_{j}(F,\tau)=\lim_{T \to \infty }L_{-1/2-2j,\tau}^{(k-1)/2-j}\int_{M_{T}}R_{-2k,z}^{k-1-2j}F(z)\Theta(\tau,z,\psi_{M,2j+1})y^{-4j-2}d\mu(z)
		\end{align*}
		for $0 \leq j \leq (k-1)/2$. Note that $-\frac{1}{2\sqrt{N}}k(k+1)I_{0} = \LambdaM$ and $I_{(k-1)/2} = \IM$. As above we see that
		\[
		I_{j} = -\frac{\pi}{N}(k-(2j+1))(k+(2j+1)+1)I_{j+1}.
		\]
		The formula for odd $k$ now follows inductively.
	\end{proof}

	\section{The Fourier expansion of $\IM(F,\tau)$}\label{sec:fourexp}

	Let $k \in \Z_{\geq 0}$ and let $F \in H_{-2k}^{+}(\Gamma)$ be a harmonic weak Maass form of weight $-2k$ for $\Gamma$. In order to describe the Fourier expansion of $\IM(F,\tau)$ we first have to introduce the modular trace function and geodesic cycle integrals.

	\subsection{Heegner points and the modular trace function}
	
	For $X\in V$ with $Q(X) = m \in \Q_{> 0}$ we let
\[
D_{X}= \mathrm{span}(X) \in D
\]
be the Heegner point of discriminant $m$ associated to $X$. We use the same symbol for the image of $D_{X}$ in $M$. Note that for $m\in \Q_{>0}$ and $h\in L'/L$ with $Q(h) \equiv m (\Z)$, the group $\G$ acts on the set 
\[
 L_{m,h}=\{ X \in L+h\,:\, Q(X)=m \}
\]
with finitely many orbits, and the stabilizer $\Gamma_{X}$ is finite.

For some $\Gamma$-invariant function $F$ function on $\h$ we define the modular trace function of $F$ by
\[
 \mt(F;m,h)=\sum_{X\in \G \bs L_{m,h}}\frac{1}{|\overline{\G}_X|} F(D_X),
\]
where $\overline{\G}_X$ denotes the stabilizer of $X$ in $\overline{\G}$, the image of $\G$ in $\mathrm{PSL}_2(\Z)$.
Moreover, we define
\[
	L^{+}_{m,h}=\left\{X=\begin{pmatrix} x_2 & x_1\\x_3&-x_2\end{pmatrix}\in L_{m,h}\,:\, x_1\geq0\right\}\quad\text{and}\quad L^{-}_{m,h}=L_{m,h}\setminus L_{m,h}^{+},
\]
and accordingly
\[
	\mt^+(F;m,h)=\sum_{X\in \G \bs L^+_{m,h}}\frac{1}{|\overline{\G}_X|} F(D_X)\quad\text{and}\quad \mt^-(F;m,h)=\sum_{X\in \G \bs L^-_{m,h}}\frac{1}{|\overline{\G}_X|}F(D_X).
\]

\subsection{Geodesic cycle integrals}

A vector $X \in V$ of negative length $Q(X)=m\in\Q_{<0}$ defines a geodesic $c_X$ in $D$ via
\[
 c_X=\{z\in D\,:\, z\perp X\}.
\]
We write $c(X)=\G_X\setminus c_X$ for the image in $M = \Gamma \setminus \h$. 

If $|m|/N$ is not a square in $\Q$, then $X^\perp$ is non-split over $\Q$ and the stabilizer $\overline{\G}_X$ is infinite cyclic. On the other hand, if $|m|/N$ is a square, then $X^\perp$ is split and $\overline{\G}_X$ is trivial.
In the first case the geodesic $c(X)$ is closed, while in the second case $c(X)$ is an infinite geodesic (see also \cite[Lemma 3.6]{funke}).

In the case that $c(X)$ is an infinite geodesic, $X$ is orthogonal to two isotropic lines $\ell_X =\operatorname{span}(Y)$ and $\widetilde\ell_X =\operatorname{span}(\widetilde{Y})$, with $Y$ and $\widetilde{Y}$ positively oriented. We call $\ell_X$ the line associated to $X$ if the triple $(X,Y,\widetilde{Y})$ is a positively oriented basis for $V$, and we write $X\sim  \ell_X$. Note that $\widetilde{\ell}_X=\ell_{-X}$.

For $m \in \Q_{<0}$ and $X \in L_{m,h}$ we define the cycle integral of a cusp form $G \in S_{2k+2}(\Gamma)$ along the geodesic $c(X)$ by
\[
\mathcal{C}(G,X) = \int_{c(X)}G(z)Q_{X}^{k}(z)dz,
\]
where the orientation of $c(X)$ is defined using an explicit parametrization as follows:

Since $Q(X) = m < 0$, there is some matrix $g \in \SL_{2}(\R)$ such that
$
g^{-1}X = \sqrt{|m|/N}\left(\begin{smallmatrix}1 & 0 \\ 0 & -1 \end{smallmatrix}\right).
$
Recall that the stabilizer $\overline{\Gamma}_{X}$ is either trivial or infinite cyclic. In the second case, the stabilizer of  $g^{-1}X$ in $g^{-1}\overline{\Gamma}$ is generated by some matrix $\left(\begin{smallmatrix}\varepsilon & 0 \\ 0  & \varepsilon^{-1}\end{smallmatrix}\right)$ with $\varepsilon > 1$. We can now parametrize $c(X)$ by $g.iy$ with $y \in (0,\infty)$ if $|m|/N$ is a square, and $y \in (1,\varepsilon^{2})$ if $|m|/N$ is not a square. Note that $\frac{d}{dy}g.iy = i\cdot j(g,iy)^{-2}$ and
\[
Q_{X}(g.iy) = j(g,iy)^{-2}Q_{g^{-1}.X}(iy) = j(g,iy)^{-2} (-2\sqrt{|m|N}iy).
\]
Writing $G_{g} = G|_{2k+2}g$ we find
\[
\mathcal{C}(G,X) = (-2\sqrt{|m|N}i)^{k}i\int_{0}^{\infty}G_{g}(iy)y^{k}dy,
\]
if $|m|/N$ is a square and similarly (i.e.~with the integral from $1$ to $\varepsilon^{2}$) if $|m|/N$ is not a square. Using the transformation behaviour of $G$ it is easy to see that the right-hand side, and thus the implied orientation of $c(X)$, is independent of the choice of the matrix $g$.

Finally, we define the trace of $G$ for $m > 0$ by
\[
\mt(G;m,h) = \sum_{X \in \Gamma \setminus L_{m,h}}\mathcal{C}(G,X).
\]

	\subsection{The complementary trace}
	
	Let $m \in \Q_{<0}$ and assume that $|m|/N$ is a square, i.e.~$m = -Nd^{2}$ for some $d \in \Q$. Let $F \in H_{-2k}^{+}(\Gamma)$. For an isotropic line $\ell$ we let $a_{\ell}^{+}(w)$ be the coefficients of the holomorphic part $F_{\ell}^{+}$ of $F_{\ell} = F|_{-2k}\sigma_{\ell}$. Let $X \in L_{-Nd^{2},h}$. Recall that $\overline{\Gamma}_{X}$ is trivial and $X$ gives rise to an infinite geodesic $c(X)$. Choosing the orientation of $V$ appropriately, we have
	\[
	\sigma_{\ell_{X}}^{-1}X = d\begin{pmatrix}1 & -2r_{\ell_{X}} \\ 0 & -1 \end{pmatrix}
	\]
	for some $r_{\ell_{X}} \in \Q$. Note that the geodesic $c_{X}$ in $D$ is given by
	\[
	c_{X} = \sigma_{\ell_{X}}\{z \in D: \Re(z) = r_{\ell_{X}}\}.
	\]
	Therefore we call $\Re(c(X)) := r_{\ell_{X}}$ the real part of $c(X)$.
	We now define the complementary trace of $F$ by
	\begin{align*}
	\mt^{c}(F;-Nd^{2},h) &= \sum_{X \in \Gamma \setminus L_{-Nd^{2},h}}\bigg(\sum_{w < 0}a_{\ell_{X}}^{+}(w)(4\pi w)^{k}e^{2\pi i \Re(c(X))w} \\
	& \qquad \qquad \qquad \qquad +(-1)^{k+1}\sum_{w < 0}a_{\ell_{-X}}^{+}(w)(4\pi w)^{k}e^{2\pi i \Re(c(-X))w}\bigg).
	\end{align*}
	
	\subsection{The Fourier expansion}
	
	We are now ready to state the Fourier expansion of the Millson theta lift.

	\begin{theorem}\label{thm:fourierexpansion}
	Let $k \in \Z_{\geq 0}$ and let $F \in H_{-2k}^{+}(\Gamma)$. For $k > 0$ the $h$-th component of $\IM(F,\tau)$ is given by
	\begin{align*}
	&\sum_{m > 0}\frac{1}{2\sqrt{m}}\left(\frac{\sqrt{N}}{4\pi \sqrt{m}}\right)^{k}\big(\mt^{+}(R_{-2k}^{k}F;m,h) + (-1)^{k+1}\mt^{-}(R_{-2k}^{k}F;m,h)\big)q^{m} \\
	&\quad +\sum_{d > 0}\frac{1}{2i\sqrt{N}d}\left(\frac{1}{4\pi i d}\right)^{k}\mt^{c}(F;-Nd^{2},h)q^{-Nd^{2}} \\
	& \quad +\frac{(-1)^{k}k!}{2\sqrt{N}\pi^{k+1}}\sum_{\substack{\ell \in \Gamma \setminus \Iso(V) \\ \ell \cap (L + h) \neq \emptyset}}a_{\ell}^{+}(0)\frac{\alpha_{\ell}}{\beta_{\ell}^{k+1}}\big(\zeta(s+1,k_{\ell}/\beta_{\ell}) + (-1)^{k+1}\zeta(s+1,1-k_{\ell}/\beta_{\ell}) \big)\big|_{s = k} \\
	&\quad-\sum_{m < 0}\frac{1}{2(4\pi|m|)^{k+1/2}}\overline{\mt(\xi_{-2k}F;m,h)}\Gamma\left(\tfrac{1}{2}-k,4\pi|m|v\right)q^{m},
	\end{align*}
	where $\zeta(s,\rho) = \sum_{n \geq 0, n+\rho \neq 0}(n+\rho)^{-s}$ is the Hurwitz zeta function, and $k_{\ell} \in \Q$ with $0 \leq k_{\ell} < \beta_{\ell}$ is defined by $\sigma_{\ell}^{-1}h_{\ell} = \left(\begin{smallmatrix}0 & k_{\ell} \\ 0 & 0 \end{smallmatrix}\right)$ for some $h_{\ell} \in \ell \cap (L+h)$.
	
	For $k = 0$ the $h$-th component of $\IM(F,\tau)$ is given by the same formula as above but with the additional non-holomorphic terms
	\begin{align*}
	&\sum_{d > 0}\frac{1}{4id\sqrt{\pi N}}\sum_{X \in \Gamma \setminus L_{-Nd^{2},h}}\big(a_{\ell_{X}}^{+}(0)-a_{\ell_{-X}}^{+}(0)\big)\Gamma\left(\tfrac{1}{2},4\pi Nd^{2}v\right)q^{-Nd^{2}}.
	\end{align*}
	\end{theorem}
	
	The Hurwitz zeta function $\zeta(s,\rho)$ is holomorphic for $\Re(s) > 1$ and has a simple pole at $s = 1$ with residue $1$ and constant term $-\psi(\rho)$, where $\psi(0) = -\gamma$ and $\psi(\rho) = \frac{\Gamma'(\rho)}{\Gamma(\rho)}$ is the digamma function if $\rho > 0$. Note that $k_{\ell} = 0$ is equivalent to $h \in L$. Thus for $k > 0$ we can simply plug in $s = k$ in the third line of the theorem, and for $k = 0$ we get
	\[
	\big(\zeta(s+1,k_{\ell}/\beta_{\ell}) -\zeta(s+1,1-k_{\ell}/\beta_{\ell}) \big)\big|_{s = 0} = \psi(1-k_{\ell}/\beta_{\ell})-\psi(k_{\ell}/\beta_{\ell}) = \begin{cases}
	0, & \text{$h \in L$,} \\
	\pi \cot(\pi k_{\ell}/\beta_{\ell}), & \text{$h \notin L$.}
\end{cases}
	\]
	
	Note that the first three lines in Theorem \ref{thm:fourierexpansion} are the holomorphic part of $\IM(F,\tau)$, whereas the fourth line and the additional terms (for $k = 0$) are the non-holomorphic part of $\IM(F,\tau)$. The alternative form of the complementary trace given in \cite[Proposition 4.6.]{brfu06} shows that the principal part of $\IM(F,\tau)$ is finite. In particular, this completes the proof of Theorem \ref{thm:PropertiesLifts}.

	For the sake of completeness we also state the Fourier expansion of the Shintani lift in our normalization.

	\begin{theorem}
		Let $k \in \Z_{\geq 0}$ and $G \in S_{2k+2}(\Gamma)$. Then the $h$-th component of $\ISh(G,\tau)$ is given by
		\[
		\ISh(G,\tau)_{h} = -\sqrt{N}\sum_{m > 0}\mt(G;-m,h)q^{m}.
		\]
	\end{theorem}
	
	\begin{proof}
		The Fourier expansion of $\ISh(G,\tau)$ can be computed using very similar, but much easier calculations as in the proof of Theorem \ref{thm:fourierexpansion} below.
	\end{proof}
	
	\subsection{Fourier coefficients of positive index}
	To compute the coefficients of positive index $m > 0$ we use the relation between $I^{M}(F,\tau)$ and $\Lambda^{M}(F,\tau)$. In the case that $k$ is odd, these coefficients were already computed in \cite{brono2,alfes}. In the case that $k$ is even, the $(m,h)$-th coefficient of $\Lambda^{M}(F,\tau)$ is given by
	\[
		C(m,h)=\sum_{X \in L_{m,h}}\int_M R_{-2k,z}^kF(z) \psi_{M}^0(X,\tau,z)d\mu(z),
	\]
	where $ \psi_{M}^0(X,\tau,z)=vp_z(X)e^{-2\pi v R(X,z)}$. For $X = \left(\begin{smallmatrix}x_{2} & x_{1} \\ x_{3} & -x_{2}\end{smallmatrix} \right) \in L_{m,h}$ and $m > 0$ we have $x_{3} \neq 0$ and
	\[
	-2\pi v R(X,z) = 2\pi v (X,X) - \pi v\left(\frac{N(x_{3}x-x_{1})^{2}+q(X)}{\sqrt{N}x_{3}y} + \sqrt{N}x_{3}y \right)^{2},
	\]
	which implies that $\psi_{M}^{0}(X,\tau,z)$ is of square-exponential decay in all directions of $\h$. Thus the integral in $C(m,h)$ actually converges without regularisation. By the usual unfolding argument we obtain
	\begin{align*}
		C(m,h)&= \sum_{X\in\G\setminus L^+_{m,h}\cup L^-_{m,h}}\frac{1}{\abs{\overline\G_X}} \int_M R_{-2k,z}^k F(z)\psi_{M}^0(X,\tau, z)d\mu(z)\\
		&=\sum_{X\in\G\setminus L^+_{m,h}}\frac{1}{\abs{\overline\G_X}} \int_M R_{-2k,z}^k F(z)\psi_{M}^0(X,\tau, z)d\mu(z)\\
		&\quad-\sum_{X\in\G\setminus L^-_{m,h}} \frac{1}{\abs{\overline\G_{-X}}} \int_M R_{-2k,z}^k F(z)\psi_M^0(-X,\tau, z)d\mu(z).
\end{align*}
Following Katok and Sarnak \cite{ks} we rewrite this as an integral over $\SL_2(\R)$ and obtain that $C(m,h)$ equals
	\[
		\int_{\SL_2(\R)}R_{-2k,z}^k F(gi) \left(\sum_{X\in\G\setminus L^+_{m,h}}\frac{1}{\abs{\overline\G_X}} \psi_{M}^0(X,\tau, gi)-\sum_{X\in\G\setminus L^-_{m,h}} \frac{1}{\abs{\overline\G_{-X}}}\psi_M^0(-X,\tau, gi) \right)dg.
	\]
Here, we normalize the Haar measure such that the maximal compact open subgroup has volume $1$.

Since the group $\SL_2(\R)$ acts transitively on $L^+_{m,h}$, there is a $g_1\in \mathrm{SL}_2(\R)$ such that $g_1^{-1}.X=\sqrt{2m} X_{1}(i)$ for $X\in L_{m,h}^+$. Also, there is a $g_1\in \mathrm{SL}_2(\R)$ such that $g_1^{-1}.(-X)=\sqrt{2m} X_{1}(i)$ for $X\in L_{m,h}^-$. We then have 
\begin{align*}
C(m,h)&=\sum_{X\in\G\setminus L^+_{m,h}} \frac{1}{\abs{\overline\G_X}} \int_{\SL_{2}(\R)} R_{-2k,z}^k F(g_1gi)\psi_{M}^0\left(\sqrt{2m}g^{-1}.X_{1}(i),\tau,i\right)dg\\
&\quad-\sum_{X\in\G\setminus L^-_{m,h}} \frac{1}{\abs{\overline\G_{-X}}} \int_{\SL_{2}(\R)}R_{-2k,z}^k F(g_1gi)\psi_{M}^0\left(\sqrt{2m}g^{-1}.X_{1}(i),\tau,i\right)dg.
\end{align*}
Then, $g_1 i$ is the Heegner point corresponding to $D_X$. 

Using the Cartan decomposition of $\SL_2(\R)$ we find 
\begin{align*}
C(m,h)=\sum_{X\in\G\setminus L^+_{m,h}} \frac{1}{\abs{\overline\G_X}} R_{-2k,z}^k F(D_X) Y_c(\sqrt{m})-\sum_{X\in\G\setminus L^-_{m,h}} \frac{1}{\abs{\overline\G_{-X}}} R_{-2k,z}^k F(D_{-X}) Y_c(\sqrt{m}),
\end{align*}
with
\[
Y_c(t)=4\pi v \int_1^\infty \psi_{M}^0(\sqrt{2}t\alpha(a)^{-1}.X_{1}(i),\tau,i) \omega_c(\alpha(a)) \frac{a^2-a^{-2}}{2}\frac{da}{a}.
\]
Here, $\omega_c(\alpha(a))=\omega_c\left(\frac{a^2+a^{-2}}{2}\right)$ is the spherical function of eigenvalue $c=-k(k+1)$ given by the Legendre polynomial $P_k(x)$ and $\alpha(a)=\left(\begin{smallmatrix} a&0\\0&a^{-1}\end{smallmatrix}\right)$. By substituting $a=e^{r/2}$ we obtain
\[
Y_c(t)=4\pi v t\int_0^\infty \cosh(r)\sinh(r)P_k(\cosh(r))e^{-4\pi v t^2\sinh(r)^2}dr.
\]
Setting $x=\sinh(r)^2$ we get
\[
Y_c(t)=2\pi v t \int_0^\infty P_k(\sqrt{1+x}) e^{-4\pi v t^2 x}dx.
\]
This is a Laplace transformation computed in equation (7) on page 180 in \cite{tables} and we obtain
\[
 Y_c(\sqrt{m}) e^{2\pi i m \tau} = \frac{1}{2\sqrt{m}} \mathcal{W}_{\frac{k}{2}+\frac34,\frac12}(4\pi m v) e(mx),
\]
where $\mathcal{W}_{s,k}(y)=y^{-k/2} W_{k/2,s-1/2}(y)$ ($y>0$) is the $\mathcal{W}$-Whittaker function.
Using (13.1.33) and (13.4.23) in \cite{Pocket} it is easy to show that
\begin{align*}
&L_{1/2}^{k/2} \left(\mathcal{W}_{\frac{k}{2}+\frac34,\frac12}(4\pi m v) e(mx)\right)\\
&= \left(\frac{1}{4\pi m}\right)^{k/2}\prod_{j=0}^{k/2-1} \left(\frac{k+1}{2}+j\right)\left(j-\frac{k}{2}\right)  \mathcal{W}_{\frac{k}{2}+\frac34,\frac12-k}(4\pi m v) e(mx)
\\
&= \left(\frac{1}{4\pi m}\right)^{k/2}\prod_{j=0}^{k/2-1} \left(\frac{k+1}{2}+j\right)\left(j-\frac{k}{2}\right)  e^{2\pi i m \tau}.
\end{align*}
Therefore, we have that 
\[
C(m,h)=  \frac{1}{2\sqrt{m}}\left(\frac{1}{4\pi m}\right)^{k/2}\prod_{j=0}^{k/2-1} \left(\frac{k+1}{2}+j\right)\left(j-\frac{k}{2}\right)  \left(\mt^+(F;m,h)-\mt^-(F;m,h) \right).
\]
Combining this with Theorem \ref{thm:relationlifts} we obtain the formula for the coefficients of positive index.

	\subsection{Fourier coefficients of negative index}

	For $m < 0$ the $(m,h)$-th coefficient of $I^{M}(F,\tau)$ is given by
	\[
	C(m,h) = \sum_{X \in \Gamma \setminus L_{m,h}}\lim_{T \to \infty}\int_{M_{T}}F(z)\sum_{\gamma \in \Gamma_{X} \setminus \Gamma}\psi_{M,k}^{0}(\gamma X,\tau,z)y^{-2k}d\mu(z),
	\]
	where
	\[
	\psi_{M,k}^{0}(X,\tau,z) = v^{k+1}p_{z}(X)Q_{X}^{k}(\bar{z})e^{-2\pi v R(X,z)}.
	\]
	We compute the individual summands for fixed $X \in L_{m,h}$. 
	
	The computation follows similar arguments as in the proof of Theorem 4.5 in \cite{brfu06}. First, a short calulculation using the rules~\eqref{eq:diffpQR} shows that the function
	\begin{equation}\label{eq:eta}
	\eta(X,\tau,z) = C_{k}v^{k+1}Q_{X}^{-k-1}(z)\frac{\partial^{k}}{\partial v^{k}}\left(v^{-1}e^{-2\pi v R(X,z)}\right), \qquad C_{k} = \frac{\sqrt{N}(2N)^{k}}{(-2\pi)^{k+1}},
	\end{equation}
	satisfies
	\[
	\xi_{2k+2,z}\eta(X,\tau,z) = \overline{\psi_{M,k}^{0}(X,\tau,z)}.
	\]
	Using Stokes' Theorem in the form given in \cite[Lemma 2.1.]{brikavia}, we obtain
	\begin{align}
	&\lim_{T \to \infty}\int_{M_{T}}F(z)\sum_{\gamma \in \Gamma_{X} \setminus \Gamma}\psi_{M,k}^{0}(\gamma X,\tau,z)y^{-2k}d\mu(z) \\
	&= -\lim_{T \to \infty}\int_{M_{T}}\overline{\xi_{-2k,z}F(z)}\sum_{\gamma \in \Gamma_{X} \setminus\Gamma}\eta(\gamma X,\tau,z)y^{2k+2}d\mu(z) \label{eq:negativecoefficienttwodimensional} \\
	&  \quad -\lim_{T\to \infty}\int_{\partial M_{T}}F(z)\sum_{\gamma \in \Gamma_{X}\setminus\Gamma}\eta(\gamma X,\tau,z)dz.\label{eq:boundaryintegral}
	\end{align}
	Since $\xi_{-2k,z}F$ is a cusp form, we can write the limit of the first integral on the right-hand side as an integral over $M$.
	
	\subsubsection{The integral over $M$}
	
	We first compute the complex conjugate of the integral over $M$ on the right-hand side. Since $Q(X) = m < 0$ we can find some matrix $g \in \SL_{2}(\R)$ such that
	\[
	X' := g^{-1}.X = \sqrt{\frac{|m|}{N}}\begin{pmatrix}1 & 0 \\ 0 & -1 \end{pmatrix}.
	\]
	Replacing $z$ by $gz$ and using the unfolding argument, we find
	\begin{align*}
	-\int_{M}\xi_{-2k,z}F(z)\overline{\sum_{\gamma \in \Gamma_{X} \setminus \Gamma}\eta(\gamma X,\tau,z)}y^{2k+2}d\mu(z) = -\int_{\Gamma_{X'}\setminus \h }\xi_{-2k,z}F_{g}(z)\overline{\eta(X',\tau,z)}y^{2k+2}d\mu(z),
	\end{align*}
	where $F_{g} = F|_{-2k}g$. 
	
	If $|m|/N$ is not a square then $\overline{\Gamma}_{X}$ is infinite cyclic and
	\[
	\Gamma_{X'} = g^{-1}\Gamma_{X}g = \left\{\pm\begin{pmatrix}\varepsilon & 0 \\0 & \varepsilon^{-1} \end{pmatrix}^{n}: n\in \Z\right\}
	\]
	for some $\varepsilon > 1$. On the other hand, if $|m|/N$ is a square then $\overline{\Gamma}_{X'}$ is trivial, so $\Gamma_{X'}\setminus \h = \h$. 
	Here we only consider the non-square case since the other case is very similar. As a fundamental domain for $\Gamma_{X'}\setminus \h$ we can take the horizontal strip $\{z \in \h: 1 \leq y < \varepsilon^{2}\}$. Using the explicit formula
	\begin{align}\label{eq:etaexplicitnegative}
	\overline{\eta(X',\tau,z)} = C_{k}v^{k+1}\left( -2\sqrt{|m|N}\bar{z} \right)^{-k-1}\frac{\partial^{k}}{\partial v^{k}}\left(v^{-1}e^{-4\pi |m|v\left(\frac{x^{2}}{y^{2}}+1\right)}\right),
	\end{align} 
	and replacing $\frac{x}{y}$ by $t$ in the integral over $x$, we find that the complex conjugate of (\ref{eq:negativecoefficienttwodimensional}) equals
	\begin{align*}
	\frac{-C_{k}v^{k+1}}{( -2\sqrt{|m|N})^{k+1}}\frac{\partial^{k}}{\partial v^{k}}v^{-1}\int_{-\infty}^{\infty}\left(\int_{1}^{\varepsilon^{2}}(t+i)^{k+1}y^{k}\xi_{-2k,z}F_{g}(y(t+i)) dy\right) \frac{e^{-4\pi |m|v\left(t^{2}+1\right)}}{(t^{2}+1)^{k+1}}  dt.
	\end{align*}
	The inner integral is the contour integral of the holomorphic function $z^{k}\xi_{-2k,z}F_{g}(z)$ along the line $y(t+i), y \in (1,\varepsilon^{2})$. 
	Using $\xi_{-2k,z}F_{g}(\varepsilon^{2}z)=\varepsilon^{-2k-2}\xi_{-2k,z}F_{g}(z)$ it is easily seen by Cauchy's Theorem that the inner integral does in fact not depend on $t$. Thus the double integral simplifies to
	\[
	\frac{-C_{k}v^{k+1}i^{k+1}}{( -2\sqrt{|m|N})^{k+1}}\left(\int_{1}^{\varepsilon^{2}}\xi_{-2k,z}F_{g}(iy)y^{k}dy\right)\frac{\partial^{k}}{\partial v^{k}}v^{-1} \int_{-\infty}^{\infty}\frac{e^{-4\pi |m|v\left(t^{2}+1\right)}}{(t^{2}+1)^{k+1}}  dt.
	\]
%
	
	It remains to compute the derivative of the last integral. If we replace $t^{2}$ by $u$ we see that the integral is equal to
	\[
	\int_{0}^{\infty}u^{-1/2}(u+1)^{-k-1}e^{-4\pi |m|v(u+1)}du = \Gamma(\tfrac{1}{2})e^{-4\pi |m|v}U\left(\tfrac{1}{2},\tfrac{1}{2}-k,4\pi|m|v\right)
	\]
	with Kummer's function $U(a,b,z)$, see \cite[13.2.5]{Pocket}. The derivative will be computed in the following, slightly more general lemma.
	
	\begin{lemma}
		For $k,\ell \in \Z, k \geq 0,$ we have
		\[
		\frac{\partial^{k}}{\partial v^{k}}v^{-1-\ell}e^{-v}U\left(\tfrac{1}{2},\tfrac{1}{2}-k-\ell,v \right) = (-1)^{k}v^{-1-k-\ell}e^{-v}U\left(\tfrac{1}{2}-k,\tfrac{1}{2}-k-\ell,v \right).
		\]
	\end{lemma}
	
	\begin{proof}
		By induction on $k$. The case $k = 0$ is clear. Suppose that the claim holds for some fixed $k$ and all $\ell$. Computing the innermost derivative by the product rule and using the recurrence relation \cite[13.4.25]{Pocket} on the $U'$ summand we get
		\begin{align*}
		&\frac{\partial^{k+1}}{\partial v^{k+1}}v^{-1-\ell}e^{-v}U\left(\tfrac{1}{2},\tfrac{1}{2}-(k+1)-\ell,v \right) \\
		&= -\frac{\partial^{k}}{\partial v^{k}}\bigg((1+\ell)v^{-1-(\ell+1)}e^{-v}U\left(\tfrac{1}{2},\tfrac{1}{2}-k-(\ell+1),v \right)+v^{-1-\ell}e^{-v}U\left(\tfrac{1}{2},\tfrac{1}{2}-k-\ell,v \right)\bigg).
		\end{align*}
		If we apply the induction hypothesis on both summands, for $\ell+1$ and $\ell$, the last line equals
		\begin{align*}
		(-1)^{k+1}v^{-1-(k+1)-\ell}e^{-v}\left((1+\ell)U\left(\tfrac{1}{2}-k,\tfrac{1}{2}-k-(\ell+1),v \right) + v U\left(\tfrac{1}{2}-k,\tfrac{1}{2}-k-\ell,v \right)\right),
		\end{align*}
		which by \cite[13.4.18]{Pocket} is the same as
		\[
		(-1)^{k+1}v^{-1-(k+1)-\ell}e^{-v}U\left(\tfrac{1}{2}-(k+1),\tfrac{1}{2}-(k+1)-\ell,v \right).
		\]
		Since $\ell$ was arbitrary, this completes the induction.
	\end{proof}

	For $\ell = 0$ the lemma gives
	\begin{align*}
	v^{k+1}\frac{\partial^{k}}{\partial v^{k}}v^{-1}\Gamma(\tfrac{1}{2})e^{-4\pi |m|v}U\left(\tfrac{1}{2},\tfrac{1}{2}-k,4\pi|m|v \right) &= (-1)^{k}\sqrt{\pi}e^{-4\pi|m|v}U\left(\tfrac{1}{2}-k,\tfrac{1}{2}-k,4\pi|m|v \right).
	\end{align*}
	By \cite[13.6.28.]{mr0167642}, the right-hand side is $(-1)^{k}\sqrt{\pi}\Gamma\left(\tfrac{1}{2}+k,4\pi|m|v\right)$. If we put everything together and recall the definition of the cycle integral of $\mathcal{C}(\xi_{-2k,z}F,X)$, we see that (\ref{eq:negativecoefficienttwodimensional}) equals
	\begin{align*}
	-\frac{1}{2(4\pi|m|)^{k+1/2}}\overline{\mathcal{C}(\xi_{-2k,z}F,X)} \Gamma(\tfrac{1}{2}+k,4\pi |m|v).
	\end{align*}

	\subsubsection{The boundary integral}

	We now consider the limit of the boundary integral in~\eqref{eq:boundaryintegral}. By the definition of the truncated curve $M_{T}$ we find
	\begin{align*}
	-\int_{\partial M_{T}}F(z)\sum_{\gamma \in \Gamma_{X}\setminus\Gamma}\eta(\gamma X,\tau,z)dz = \sum_{\ell \in \Gamma \setminus \Iso(V)}\int_{z = iT}^{\alpha_{\ell}+iT}F_{\ell}(z)\sum_{\gamma \in \Gamma_{X} \setminus\Gamma}\eta(\sigma_{\ell}^{-1}\gamma X,\tau,z)dz,
	\end{align*}
	where $F_{\ell} = F|_{-2k}\sigma_{\ell}$. As in the proof of Lemma 5.2 in \cite{brfu06} we see that for each isotropic line $\ell$ the integral vanishes in the limit unless $X$ is orthogonal to $\ell$ and $\gamma \in \Gamma_{\ell}$, which can only happen if $|m|/N$ is a square and $\ell = \ell_{X}$ or $\ell= \ell_{-X}$. 
	
	In particular, if $|m|/N$ is not a square, then the whole boundary integral vanishes.
	
	On the other hand, if $|m|/N$ is a square, we obtain
	\begin{align}
	\label{eq:fourcoeff1}-\int_{\partial M_{T}}F(z)\sum_{\gamma \in \Gamma_{X}\setminus\Gamma}\eta(\gamma X,\tau,z)dz &= \int_{z = iT}^{\alpha_{\ell_{X}}+iT}F_{\ell_{X}}(z)\sum_{\gamma \in \overline{\Gamma}_{\ell_{X}}}\eta(\sigma_{\ell_{X}}^{-1}\gamma X,\tau,z)dz, \\
	& \quad  +\int_{z = iT}^{\alpha_{\ell_{-X}}+iT}F_{\ell_{-X}}(z)\sum_{\gamma \in \overline{\Gamma}_{\ell_{-X}}}\eta(\sigma_{\ell_{-X}}^{-1}\gamma X,\tau,z)dz.
	\end{align}
	We only compute the first integral on the right-hand side since the second one can be computed in the same way if we first write
	\[
	\eta(\sigma_{\ell_{-X}}^{-1}\gamma X,\tau,z) = (-1)^{k+1}\eta(\sigma_{\ell_{-X}}^{-1}\gamma (-X),\tau,z).
	\]
	
	Let $\ell = \ell_{X}$ for brevity. Choosing the orientation of $V$ appropriately, we can assume that
	\[
	X' := \sigma_{\ell}^{-1}.X = \sqrt{\frac{|m|}{N}}\begin{pmatrix}1 & -2r_{\ell} \\ 0 & -1 \end{pmatrix}
	\]
	for some $r_{\ell} \in \Q$. Then the first summand in \eqref{eq:fourcoeff1} equals
	\[
	\int_{z = iT}^{\alpha_{\ell}+iT}F_{\ell}(z)\sum_{\gamma \in \sigma_{\ell}^{-1}\overline{\Gamma}_{\ell}\sigma_{\ell}}\eta(\gamma X',\tau,z)dz.
	\]
	Recall that $\sigma_{\ell}^{-1}\overline{\Gamma}_{\ell}\sigma_{\ell}$ consists of the matrices $\left( \begin{smallmatrix}1 & \alpha_{\ell}n \\ 0 & 1 \end{smallmatrix}\right)$ with $n \in \Z$. Using the definition of $\eta$, we see that the first summand in \eqref{eq:fourcoeff1} equals
	\begin{align*}
	& \int_{z = iT}^{\alpha_{\ell}+iT}F_{\ell}(z)\sum_{n \in \Z}\eta\left(\sqrt{\frac{|m|}{N}}\begin{pmatrix} 1 & 2(\alpha_{\ell}n-r_{\ell}) \\ 0 & -1\end{pmatrix},\tau,z\right)dz \\
	&= \frac{C_{k}v^{k+1}}{(-2\sqrt{|m|N})^{k+1}}\frac{\partial^{k}}{\partial v^{k}}v^{-1}e^{-4\pi|m|v}\int_{z = iT}^{\alpha_{\ell}+iT}F_{\ell}(z)\sum_{n \in \Z} (z+\alpha_{\ell} n-r_{\ell})^{-k-1}e^{-4\pi |m|v\frac{(x+\alpha_{\ell} n - r_{\ell})^{2}}{y^{2}}} dz.
	\end{align*}
	For a function $g(t)$ on $\R$ we let $\hat{g}(w) = \int_{-\infty}^{\infty}g(t)e^{2\pi i tw}dt$ be its Fourier transform. Using Poisson summation we can rewrite the inner sum as
	\begin{align*}
	&\sum_{n \in \Z}(z+\alpha_{\ell} n - r_{\ell})^{-k-1}e^{-4\pi |m|v\frac{(x+\alpha_{\ell} n - r_{\ell})^{2}}{y^{2}}} \\
	&= \frac{1}{\alpha_{\ell}}\sum_{w \in \frac{1}{\alpha_{\ell}}\Z}e^{-2\pi i w(x-r_{\ell})}\int_{-\infty}^{\infty}(t+iy)^{-k-1}e^{-4\pi|m|v \frac{t^{2}}{y^{2}}}e^{2\pi i w t}dt,
	\end{align*}
	where we replaced $t = (x+\alpha_{\ell} n-r_{\ell})$. The required Fourier transform is computed in the next lemma. We let $a = 2\sqrt{\pi|m|v}/y$ and $b = y$ for brevity.
	
	\begin{lemma}
		For $a,b \neq 0$ and $k \in \Z,k\geq 0,$ the Fourier transform of
		\[
		h_{k}(t) = (t+ib)^{-k-1}e^{-a^{2}t^{2}}
		\]
		is given by
		\begin{align*}
		\widehat{h}_{k}(w) &= 
		 -\frac{i^{k+1}}{k!}\pi  e^{a^{2}b^{2}}e^{2\pi bw}\bigg(\erfc(ab+\pi w/a)\sum_{j=0}^{k}\binom{k}{j}(2\pi w)^{k-j}(-ia)^{j}H_{j}(iab) \\
		& \quad + e^{-(ab + \pi w/a)^{2}}\frac{2}{\sqrt{\pi}}\sum_{j=1}^{k}\binom{k}{j}(2\pi w)^{k-j}(-a)^{j}\sum_{\ell=0}^{j-1}\binom{j}{\ell}i^{\ell}H_{\ell}(iab)H_{j-\ell-1}(ab+\pi w/a)\bigg),
		\end{align*}
		where $\erfc(x) = \frac{2}{\sqrt{\pi}}\int_{x}^{\infty}e^{-u^{2}}du$ is the standard complementary error function and $H_{n}(x) = (-1)^{n}e^{x^{2}}\frac{d^{n}}{dx^{n}}e^{-x^{2}}$ is the $n$-th Hermite polynomial.
	\end{lemma}
	
	\begin{proof}
		Since
		\[
		(t+ib)^{-k-1}e^{-a^{2}t^{2}} = \frac{i^{k}}{k!} \left( \frac{\partial^{k}}{\partial b^{k}}(t+ib)^{-1}\right)e^{-a^{2}t^{2}},
		\]
		the formula for $\widehat{h}_{k}$ follows from the one for $\widehat{h}_{0}$ and Leibniz's rule. Thus it suffices to prove that the Fourier transform of
		\[
		h_{0}(t) = (t+ib)^{-1}e^{-a^{2}t^{2}} = (t-ib)\frac{e^{-a^{2}t^{2}}}{t^{2}+b^{2}}
		\]
		is given by
		\[
		\widehat{h}_{0}(w) =  -i\pi  e^{a^{2}b^{2}}e^{2\pi bw}\erfc(ab+\pi w/a).
		\]
		
		Using the well known facts that the Fourier transforms of $e^{-a^{2}t^{2}}$ and $\frac{1}{t^{2}+b^{2}}$ are given by $\frac{\sqrt{\pi}}{a}e^{-\pi^{2}w^{2}/a^{2}}$ and $\frac{\pi}{b}e^{-2\pi b|w|}$, respectively, and that the Fourier transform of a product of two functions is the convolution of the individual transforms, we see that the Fourier transform of $f(t) = \frac{e^{-a^{2}t^{2}}}{t^{2}+b^{2}}$ is given by
		\begin{align*}
		\widehat{f}(w) &= \frac{\pi^{3/2}}{ab}\int_{-\infty}^{\infty}e^{-\pi^{2}x^{2}/a^{2}}e^{-2\pi b|w-x|} dx \\
		&= 	\frac{\pi^{3/2}}{ab}e^{2\pi bw}\int_{w}^{\infty}e^{-\pi^{2}x^{2}/a^{2}}e^{-2\pi b x} dx  + \frac{\pi^{3/2}}{ab}e^{-2\pi bw}\int_{-w}^{\infty}e^{-\pi^{2}x^{2}/a^{2}}e^{-2\pi bx} dx  \\
		&= \frac{\pi}{2b}e^{a^{2}b^{2}}\left(e^{2\pi bw}\erfc(ab+\pi w/a) + e^{-2\pi bw}\erfc(ab-\pi w/a) \right).
		\end{align*}
		Since the Fourier transform of $tf(t)$ is given by $-\frac{i}{2\pi}\frac{d}{dw}\widehat{f}(w)$, we obtain
		\[
		\widehat{(tf)}(w) = -\frac{i\pi}{2}e^{a^{2}b^{2}}\left(e^{2\pi bw}\erfc(ab+\pi w/a) - e^{-2\pi bw}\erfc(ab-\pi w/a) \right).
		\]
		Using $\widehat{h}_{0} =\widehat{(tf)} - ib\widehat{f}$ we get the stated formula.
	\end{proof}
	
	Let $a_{\ell}^{+}(w)$ and $a_{\ell}^{-}(w)$ denote the Fourier coefficients of $F_{\ell}$. Using the above lemma with $a=2\sqrt{\pi|m|v}/y$ and $b= y$, we find that the right-hand side of \eqref{eq:fourcoeff1} is equal to
	\begin{align*}
	&\frac{-C_{k}v^{k+1}i^{k+1}\pi}{(-2\sqrt{|m|N})^{k+1}k!}\frac{\partial^{k}}{\partial v^{k}}v^{-1} \lim_{T \to \infty}\sum_{w \in \frac{1}{\alpha_{\ell}}\Z}(a_{\ell}^{+}(w)+a^{-}_{\ell}(w)\Gamma(1+2k,4\pi |w| T))e^{2\pi i r_{\ell}w} \\
	&\times \bigg(\erfc\left(2\sqrt{\pi|m|v}+T\sqrt{\pi} w/(2\sqrt{|m|v})\right) \\
	& \quad  \qquad \sum_{j=0}^{k}\binom{k}{j}(2\pi w)^{k-j}\left(-2i\sqrt{\pi|m|v}/T\right)^{j}H_{j}\left(2i\sqrt{\pi|m|v}\right) \\
		& \qquad + e^{-\left(2\sqrt{\pi|m|v}+T\sqrt{\pi} w/(2\sqrt{|m|v})\right)^{2}}\frac{2}{\sqrt{\pi}}\sum_{j=1}^{k}\binom{k}{j}(2\pi w)^{k-j}\left(-2\sqrt{\pi|m|v}\right)^{j} \\
		&\qquad \quad \sum_{\ell=0}^{j-1}\binom{j}{\ell}i^{\ell}H_{\ell}\left(2i\sqrt{\pi|m|v}\right)H_{j-\ell-1}\left(2\sqrt{\pi|m|v}+T\sqrt{\pi} w/(2\sqrt{|m|v})\right)\bigg).
	\end{align*}
	Note that $\erfc(x) = O(e^{-x^{2}})$ as $x \to +\infty$ and $\lim_{x \to -\infty}\erfc(x) = 2$. Further, the incomplete Gamma function is of linear exponential growth. 
	
	For $k = 0$ the last three lines disappear, and the summands for $w > 0$ vanish as $T \to \infty$. Thus all that remains in the limit is
	\begin{align*}
	-\frac{i}{2\sqrt{|m|}} \sum_{w < 0}a_{\ell}^{+}(w)e^{2\pi ir_{\ell}w} - \frac{i}{4\sqrt{|m|}} a_{\ell}^{+}(0)\erfc(2\sqrt{\pi|m|v}),
	\end{align*}
	in this case. Note that $\sqrt{\pi}\erfc(2\sqrt{\pi|m|v}) = \Gamma(\tfrac{1}{2},4\pi|m|v)$.
	
	For $k > 0$ all summands for $w \geq 0$ vanish in the limit. Further, the summands for $1 \leq j \leq k$ in the third row, and the two last rows vanish as $T \to \infty$. Thus we are left with
	\[
	\frac{(-i)^{k+1}}{2\sqrt{|m|}}\left(\frac{\sqrt{N}}{4\pi\sqrt{|m|}}\right)^{k}\sum_{w < 0}a_{\ell}^{+}(w)(4\pi w)^{k}e^{2\pi i r_{\ell}w},
	\]
	if $k > 0$. Here we used $v^{k+1}\frac{\partial^{k}}{\partial v^{k}}v^{-1} = (-1)^{k}k!$.

	\subsection{Fourier coefficients of index $0$}

	We now want to compute 
	\[
	C(0,h) = \lim_{T \to \infty}\int_{M_{T}}F(z)\sum_{X \in L_{0,h}}\psi_{M,k}^{0}(X,\tau,z)y^{-2k}d\mu(z),
	\]
	where $L_{0,h} = \{X \in L + h:Q(X) = 0\}$. Note that the sum over $X$ is now infinite. Further, we have $\psi_{M,k}^{0}(0,\tau,z) = 0$ so we can leave out the summand for $X = 0$. The computation for $Q(X) = 0$ is quite similar to the one for $Q(X) < 0$ above, so we skip some arguments. Using the function $\eta(X,\tau,z)$ defined in~\eqref{eq:eta} and Stokes' Theorem we get
	\begin{align*}
	C(0,h) &= -\lim_{T\to \infty}\int_{M_{T}}\overline{\xi_{-2k,z}F(z)}\sum_{\substack{X \in L_{0,h} \\ X \neq 0}}\eta(X,\tau,z)y^{2k+2}d\mu(z) \\
	&\quad -\lim_{T\to \infty}\int_{\partial M_{T}}F(z)\sum_{\substack{X \in  L_{0,h} \\ X \neq 0}}\eta(X,\tau,z)dz.
	\end{align*}
	Since $\xi_{-2k,z}F$ is a cusp form, we can write the first integral on the right-hand side as an integral over $M$.
	
	For each isotropic line $\ell \in \Iso(V)$ we choose a positively oriented primitive vector $X_{\ell} \in \ell \cap L$.
	If $\ell \cap (L + h) \neq \emptyset$ we can fix some vector $h_{\ell} \in \ell \cap (L+h)$ and write $\ell \cap (L+h) = \Z X_{\ell} + h_{\ell}$. Note that $\sigma_{\ell}^{-1}(nX_{\ell} + h_{\ell}) = \left(\begin{smallmatrix}0 & n\beta_{\ell} + k_{\ell} \\ 0 & 0 \end{smallmatrix}\right)$ for some $k_{\ell} \in \Q$.
	
	We now parametrize the set $L_{0,h}\setminus \{0\}$ by the points $n X_{\ell} + h_{\ell}$, where $\ell$ runs through all isotropic lines with $\ell \cap (L+h) \neq \emptyset$ and $n$ runs through $\Z$ such that $n\beta_{\ell}+k_{\ell} \neq 0$.

	\subsubsection{The integral over $M$}

	Using the above parametrization for $L_{0,h}\setminus\{0\}$ the integral over $M$ in $C(0,h)$ becomes
	\begin{align*}
	-\sum_{\substack{\ell \in \Gamma \setminus \Iso(V) \\ \ell \cap (L + h) \neq \emptyset}}\int_{M}\overline{\xi_{-2k,z}F(z)}\sum_{\substack{n \in \Z \\ n\beta_{\ell}+k_{\ell} \neq 0}}\sum_{\gamma \in \Gamma_{\ell}\setminus \Gamma}\eta(\gamma(nX_{\ell} + h_{\ell}),\tau,z)y^{2k+2}d\mu(z).
	\end{align*}
	Replacing $z$ by $\sigma_{\ell}z$ and using the unfolding argument, we get
	\begin{align*}
	-\sum_{\substack{\ell \in \Gamma \setminus \Iso(V) \\ \ell \cap (L + h) \neq \emptyset}}\int_{0}^{\infty}\int_{0}^{\alpha_{\ell}}\overline{\xi_{-2k,z}F_{\ell}(z)}\sum_{\substack{n \in \Z \\ n\beta_{\ell}+k_{\ell} \neq 0}}\eta\left(\begin{pmatrix}0 & n\beta_{\ell} + k_{\ell}  \\ 0 & 0\end{pmatrix},\tau,z\right)y^{2k+2} \frac{dx \ dy}{y^{2}}
	\end{align*}
	where $F_{\ell} = F|_{-2k}\sigma_{\ell}$. Explicitly, we have
	\begin{align}\label{eq:etaexplicit}
	\eta\left(\begin{pmatrix}0 & n\beta_{\ell} + k_{\ell}  \\ 0 & 0\end{pmatrix},\tau,z\right) = C_{k}v^{k+1}(-N(n\beta_{\ell}+k_{\ell}))^{-k-1}\frac{\partial^{k}}{\partial v^{k}}\left(v^{-1}e^{-\pi v N\frac{(n\beta_{\ell}+k_{\ell})^{2}}{y^{2}}} \right),
	\end{align}
	which is independent of $x$. Therefore the integral over $x$ picks out the constant coefficient of $\overline{\xi_{-2k,z}F_{\ell}}$, which is $0$ since $\xi_{-2k,z}F_{\ell}$ is a cusp form. Thus the integral over $M$ vanishes.
	
	\subsubsection{The boundary integral}
	
	Plugging in the definition of the truncated curve \eqref{eq:TruncFun}, the boundary integral is given by
	\begin{align*}
	\lim_{T\to \infty}\sum_{\substack{\ell \in \Gamma \setminus \Iso(V) \\ \ell \cap (L + h) \neq \emptyset}}\sum_{\ell' \in \Gamma \setminus \Iso(V)}\int_{z = iT}^{\alpha_{\ell'} + iT}F_{\ell'}(z)\sum_{\substack{n \in \Z \\ n\beta_{\ell} + k_{\ell} \neq 0}}\sum_{\gamma \in \Gamma_{\ell}\setminus\Gamma}\eta\left(\sigma_{\ell'}^{-1}\gamma \sigma_{\ell}\begin{pmatrix}0 & n\beta_{\ell} + k_{\ell}  \\ 0 & 0\end{pmatrix},\tau,z\right)dz.
	\end{align*}
	It can be seen as in the proof of Lemma 5.2 in \cite{brfu06} that in the limit only the contributions for $\ell' = \ell$ and $\gamma \in \Gamma_{\ell}$ remain, so we get
	\begin{align*}
	\lim_{T\to \infty}\sum_{\substack{\ell \in \Gamma \setminus \Iso(V) \\ \ell \cap (L + h) \neq \emptyset}}\int_{z= iT}^{\alpha_{\ell}+iT}F_{\ell}(z)\sum_{\substack{n \in \Z \\ n\beta_{\ell} + k_{\ell} \neq 0}}\eta\left(\begin{pmatrix}0 & n\beta_{\ell} + k_{\ell}  \\ 0 & 0\end{pmatrix},\tau,z\right)dz.
	\end{align*}
	Using the explicit form (\ref{eq:etaexplicit}) of $\eta$ and carrying out the integral this becomes
	\begin{align*}
	\frac{C_{k}v^{k+1}}{(-N)^{k+1}}\frac{\partial^{k}}{\partial v^{k}}v^{-1}\sum_{\substack{\ell \in \Gamma \setminus \Iso(V) \\ \ell \cap (L + h) \neq \emptyset}}\alpha_{\ell}a^{+}_{\ell}(0)\lim_{T\to \infty}\sum_{\substack{n \in \Z \\ n\beta_{\ell} + k_{\ell} \neq 0}}(n\beta_{\ell}+k_{\ell})^{-k-1}e^{-\pi vN\frac{(n\beta_{\ell}+k_{\ell})^{2}}{T^{2}}}.
	\end{align*}
	If $k_{\ell}/\beta_{\ell} \in \Z$, we can shift the summation index by $k_{\ell}/\beta_{\ell}$ and see that the terms with $n$ and $-n$ cancel if $k$ is even and add up if $k$ is odd, so in this case the limit of the sum over $n$ is $0$ if $k$ is even or $2\beta_{\ell}^{-k-1}\zeta(k+1)$ if $k$ is odd.
	On the other hand, $k_{\ell}/\beta_{\ell} \in \Z$ is only possible if $h_{\ell}$ is an integral multiple of $X_{\ell}$ and hence in $L$, i.e.~this only happens for $h = 0 \mod L$. 
	
	Now let $h \neq 0 \mod L$ and thus $k_{\ell}/\beta_{\ell} \notin\Z$. For $k > 0$ we can interchange the sum and the limit by the dominated convergence theorem. Splitting the sum into $n \geq 0$ and $n < 0$ and replacing $n$ by $1-n$ in the second part, we obtain
	\begin{align*}
	&\lim_{T\to \infty}\sum_{n \in \Z}(n\beta_{\ell}+k_{\ell})^{-k-1}e^{-\pi vN\frac{(n\beta_{\ell}+k_{\ell})^{2}}{T^{2}}} \\
	 &\quad\quad\quad\quad=\beta_{\ell}^{-k-1}\left(\zeta(k+1,k_{\ell}/\beta_{\ell}) +(-1)^{k+1}\zeta(k+1,1-k_{\ell}/\beta_{\ell})\right),
	\end{align*}
	where $\zeta(s,\rho) = \sum_{n \geq 0}(n+\rho)^{-s}$ denotes the Hurwitz zeta function. For $k = 0$ we first reorder the sum as
	\begin{align*}
	&\sum_{n \in \Z}(n\beta_{\ell}+k_{\ell})^{-k-1}e^{-\pi N v\frac{(n\beta_{\ell}+k_{\ell})^{2}}{T^{2}}}= k_{\ell}^{-1}e^{-\pi vN\frac{k_{\ell}^{2}}{T^{2}}} \\
	&\qquad \qquad +
	\beta_{\ell}^{-1}\sum_{n > 0}\left((n+k_{\ell}/\beta_{\ell})^{-1}e^{-\pi  vN\frac{(n\beta_{\ell}+k_{\ell})^{2}}{T^{2}}} + (-n+k_{\ell}/\beta_{\ell})^{-1}e^{-\pi  vN\frac{(-n\beta_{\ell}+k_{\ell})^{2}}{T^{2}}}\right).
	\end{align*}
	Now using dominated convergence again, this goes to
	\[
	\beta_{\ell}^{-1}\left(\sum_{n > 0}\left((n+k_{\ell}/\beta_{\ell})^{-1} + (-n+k_{\ell}/\beta_{\ell})^{-1}\right) +(k_{\ell}/\beta_{\ell})^{-1}\right) = \beta_{\ell}^{-1}\pi\cot(\pi k_{\ell}/\beta_{\ell})
	\]
	as $T \to \infty$. Note that $\frac{C_{k}v^{k+1}}{(-N)^{k+1}} \frac{\partial^{k}}{\partial v^{k}}v^{-1} = \frac{(-1)^{k}k!}{2\sqrt{N}\pi^{k+1}}$. This completes the calculation of $C(0,h)$.

\section{The twisted Millson theta lift}
We now explain how to obtain twisted versions of the two theta lifts and how this leads to a generalization of results by Zagier in \cite{zatr}. Throughout this section we let
\[
L= \left\{\begin{pmatrix}b&-a/N\\c&-b\end{pmatrix}\,:\,a,b,c\in\Z\right\}
\]
be the lattice given in Example~\ref{ex:lattice}, with dual lattice
\[
L'= \left\{\begin{pmatrix}b/2N&-a/N\\c&-b/2N\end{pmatrix}\,:\,a,b,c\in\Z\right\},
\]
and we let $\Gamma = \Gamma_{0}(N)$. Note that $\Gamma$ takes $L$ to itself and acts trivially on $L'/L$.

From now on we let $\Delta\in\Z$ be a fundamental discriminant and $r\in \Z$ such that $\Delta \equiv r^2  (4N)$. We consider the rescaled lattice $\Delta L$ together with the quadratic form $Q_{\Delta}(X) := \frac{1}{|\Delta|}Q(X)$. The corresponding bilinear form is given by $(\cdot,\cdot)_{\Delta} = \frac{1}{|\Delta|}(\cdot,\cdot)$, and the dual lattice of $\Delta L$ with respect to $(\cdot,\cdot)_{\Delta}$ is equal to $L'$ as above. We denote the discriminant group $L'/\Delta L$ by $\dgdelta$. Note that $\dg(1)=\dg$ and $\abs{\dgdelta}=\abs{\Delta}^3 \abs{\dg} = 2N\, \abs{\Delta}^3$.

Following \cite{gkz} we define a generalized genus character for
$\delta = \left(\begin{smallmatrix} b/2N& -a/N \\ c&-b/2N \end{smallmatrix}\right) \in L'$ by
\begin{equation*}
\chi_{\Delta}(\delta)=\chi_{\Delta}(\left[a,b,Nc\right]):=
\begin{cases}
\Deltaover{n}, & \text{if } \Delta | b^2-4Nac,\,(b^2-4Nac)/\Delta \text{ is a}
\\
& \text{square mod } 4N \text{ and } \gcd(a,b,c,\Delta)=1,
\\
0, &\text{otherwise}.
\end{cases}
\end{equation*}
Here, $\left[a,b,Nc\right]$ is the integral binary quadratic form
corresponding to $\delta$, and $n$ is any integer prime to $\Delta$ represented by one of the quadratic forms $[N_1a,b,N_2c]$ with $N_1N_2 = N$ and $N_1, N_2 > 0$. Note that the function $\chi_{\Delta}$ is invariant under the action of $\G_0(N)$.

Since $\chi_{\Delta}(\delta)$ depends only on $\delta \in L'$ modulo $\Delta L$, we can view it as a function on the discriminant group $\dgdelta$. Let $\rho_\Delta$ be the representation corresponding to $\dgdelta$. In \cite{ae} it was shown that we obtain an intertwiner of the Weil representations
corresponding to $\dg=L'/L$ and $\dgdelta$ via $\chi_\Delta$.

\begin{proposition}[{\cite[Proposition 3.2.]{ae}}] \label{prop:intertwiner}
Let $\pi: \dgdelta \rightarrow \dg$ be the natural projection.
 For $h \in \dg$, we define
 \begin{equation}
  \psi_{\Delta,r}(\e_h) := \sum_{\substack{\delta \in \dgdelta \\ \pi(\delta)=rh \\ Q_{\Delta}(\delta) \equiv \sgn(\Delta) Q(h) \, (\Z)}} \chi_\Delta(\delta) \e_\delta.
 \end{equation}
Then $\psi_{\Delta,r}: \dg \rightarrow \dgdelta$ defines an intertwining linear map between the representations $\widetilde{\rho}_{L}$ and $\rho_\Delta$, where
\[
  \widetilde{\rho}_{L} =
 \begin{cases}
 \rho_{L} & \text{if } \Delta>0, \\
		      \overline\rho_{L} & \text{if } \Delta<0.
 \end{cases}
\]
\end{proposition}
We obtain twisted versions of the Millson, Kudla-Millson and Shintani theta function introduced in Section~\ref{sec:thetafunctions} by setting
\[
\Theta_{\Delta,r}(\tau,z,\varphi) =  \sum_{h \in L'/L}\left\langle \psi_{\Delta,r}(\e_{h}),\overline{\Theta_{\dgdelta}(\tau,z,\varphi)} \right\rangle \e_{h},
\] 
where 
\[
\Theta_{\dgdelta}(\tau,z,\varphi) = \sum_{\delta \in \dgdelta}\sum_{X \in \delta + \Delta L}\varphi(X,\tau,z)\e_{\delta}
\]
is the usual theta function associated to a Schwartz function $\varphi$ and the discriminant group $\dgdelta$. It is easy to check that these twisted theta functions have the same transformation behaviour as their untwisted counterparts (see Proposition \ref{prop:propertiestheta}) and satisfy the same growth estimates (see Proposition \ref{prop:growththeta}) and differential equations if we replace $\rho_{L}$ by $\tilde{\rho}_{L}$, $N$ by $N/|\Delta|$ and $\Theta_{\ell,k}$ by $\Theta_{\Delta,r,\ell,k} = \sum_{h \in L'/L}\langle \psi_{\Delta,r}(\e_{h}),\overline{\Theta_{\dgdelta,\ell,k}}\rangle \e_{h}$.
%

Using the twisted theta functions we construct twisted analogs of the lifts considered in Section \ref{sec:thetalifts}. For example, for a harmonic weak Maass form $F \in H_{-2k}^{+}(\Gamma_{0}(N))$ we define the twisted Millson theta lift by
\[
\IM_{\Delta,r}(F,\tau) = \lim_{T \to \infty}\int_{M_{T}}F(z)\Theta_{\Delta,r}(\tau,z,\psi_{M,k})y^{-2k}d\mu(z).
\]
Analogously we obtain the twisted Shintani lift $\ISh_{\Delta,r}$. The twisted theta lifts have the same mapping properties as their untwisted versions (see Theorem \ref{thm:PropertiesLifts}), again with $\rho_{L}$ replaced by $\tilde{\rho}_{L}$. Further, we obtain the following generalization of Proposition \ref{prop:xidiagram}.

\begin{proposition}\label{prop:xidiagramtwisted}
		For $F \in H_{0}^{+}(\Gamma)$ we have
		\[
		\xi_{1/2,\tau}(\IM_{\Delta,r}(F,\tau)) = -\frac{\sqrt{|\Delta|}}{2\sqrt{N}}\ISh_{\Delta,r}(\xi_{0,z}F,\tau)+\frac{1}{2N}\sum_{\ell \in \Gamma_{0}(N) \setminus \Iso(V)}\varepsilon_{\ell}\overline{a_{\ell}^{+}(0)\Theta_{\Delta,r,\ell,1}(\tau)},
		\]
		and for $k \in \Z_{> 0}$ and $F \in H_{-2k}^{+}(\Gamma)$ we have
		\[
		\xi_{1/2-k,\tau}(\IM_{\Delta,r}(F,\tau)) = -\frac{\sqrt{|\Delta|}}{2\sqrt{N}}\ISh_{\Delta,r}(\xi_{-2k,z}F,\tau).
		\]
	\end{proposition}

	\begin{proof}
		Let us assume $k = 0$ for simplicity. Following the approach of \cite{ae} we write
		\begin{align}\label{eq:TwistedMillsonLiftUntwisted}
		\IM_{\Delta,r}(F,\tau) = \frac{1}{[\Gamma_{0}(N):\Gamma_{\Delta}]}\sum_{ h\in L'/L}\left\langle \psi_{\Delta,r}(\e_{h}),\overline{\IM(F,\tau,\dgdelta,\Gamma_{\Delta})}\right\rangle \e_{h},
		\end{align}
		where
		\[
		\IM(F,\tau,\dgdelta,\Gamma_{\Delta}) = \lim_{T\to \infty}\int_{M(\Delta)_{T}}F(z)\Theta_{\dgdelta}(\tau,z,\psi_{M,k})d\mu(z)
		\]
		is the untwisted Millson theta lift for the lattice $\Delta L$ and the group $\Gamma_{\Delta}$ consisting of all elements in $\Gamma_{0}(N)$ which act trivially on $\dgdelta$, and $M(\Delta)_{T}$ is the truncated version of the curve $M(\Delta) = \Gamma_{\Delta} \setminus \h$. By Proposition \ref{prop:xidiagram} we have
		\begin{align*}
		&\xi_{1/2,\tau}(\IM(F,\tau,\dgdelta,\Gamma_{\Delta})) \\
		&= -\frac{\sqrt{|\Delta|}}{2\sqrt{N}}\ISh(\xi_{0,z}F,\tau,\dgdelta,\Gamma_{\Delta})+\frac{|\Delta|}{2N}\sum_{\ell \in \Gamma_{\Delta} \setminus \Iso(V)}\varepsilon(\Delta)_{\ell}\overline{a_{\ell}^{+}(0)\Theta_{\dgdelta,\ell,1}(\tau)}.
		\end{align*}
		Now a short calculation, using $\varepsilon(\Delta)_{\ell} = \frac{[\Gamma_{0}(N)_{\ell}:(\Gamma_{\Delta})_{\ell}]}{|\Delta|}\varepsilon_{\ell}$, the decomposition \eqref{eq:TwistedMillsonLiftUntwisted} and the analogous decomposition for $\ISh_{\Delta,r}(\xi_{0,z}F,\tau)$, yields the result.
	\end{proof}


	This relation gives an interesting criterion for the non-vanishing of the twisted $L$-function of a newform at the critical point.
	
	\begin{theorem}
		Let $F \in H_{-2k}^{+}(\Gamma_{0}(N))$, with vanishing constant terms at all cusps if $k = 0$, such that $G =\xi_{-2k}F \in S_{2k+2}(\Gamma_{0}(N))$ is a normalized newform. For $\Delta < 0$ with $(\Delta,N) = 1$ the lift $\IM_{\Delta,r}(F,\tau)$ is weakly holomorphic if and only if $L(G,\chi_{\Delta},k+1) = 0$.
	\end{theorem}
	
	\begin{proof}
		By the last proposition, $\IM_{\Delta,r}(F,\tau)$ is weakly holomorphic if and only if the Shintani lift $\ISh_{\Delta,r}(G,\tau)$ vanishes. Since $G$ is a normalized newform, Corollary $2$ in Section II.4. of \cite{gkz} shows that the square of the absolute value of the $D$-th coefficient ($D < 0$ with $(D,N) = 1$ a fundamental discriminant) of $\ISh_{\Delta,r}(G,\tau)$ (viewed as a Jacobi form) is up to non-zero factors given by $L(G,\chi_{\Delta},k+1)L(G,\chi_{D},k+1)$. If $L(G,\chi_{\Delta},k+1) = 0$, then all fundamental coefficients of $\ISh_{\Delta,r}(G,\tau)$ vanish, which implies $\ISh_{\Delta,r}(G,\tau) = 0$. Conversely, the vanishing of the Shintani lift in particular means the vanishing of its $\Delta$-th coefficient, i.e.~$L(G,\chi_{\Delta},k+1)^{2} = 0$. This completes the proof.
	\end{proof}

To describe the Fourier coefficients of the twisted Millson lift we introduce twisted traces of CM values and cycle integrals.

If $m\in\Q_{>0}$ with $m \equiv \sgn(\Delta)Q(h)\ (\Z)$ and $h\in \dg$ we define the twisted trace of a $\Gamma_{0}(N)$-invariant function $F$ by
 \begin{align*}
 \mt^{+}_{\Delta,r}(F;m,h)&=\sum_{X\in \G_0(N)\setminus L^{+}_{\abs{\Delta}m,rh}}\frac{\chi_\Delta(X)}{\abs{\overline\G_X}} F(D_X),
\end{align*}
and $\mt^-_{\Delta,r}(F;m,h)$ accordingly.

For $m\in\Q_{<0}$ with $m \equiv \sgn(\Delta)Q(h)\ (\Z)$ and $h\in \dg$ we define the twisted trace of a cusp form $G \in S_{2k+2}(\Gamma_{0}(N))$ by
\begin{align*}
\mt_{\Delta,r}(F;m,h) = \sum_{X\in \G_0(N)\setminus L_{\abs{\Delta}m,rh}}\chi_\Delta(X) \mathcal{C}(G,X),
\end{align*}
with the cycle integral $\mathcal{C}(G,X)$ defined in Section \ref{sec:fourexp}.
 
Finally, for $m = -N|\Delta|d^{2} < 0$ with $d \in \Q_{> 0}$ we define the twisted complementary trace by
\begin{align*}
	\mt^{c}_{\Delta,r}(F;-N|\Delta|d^{2},h) &= \sum_{X \in \Gamma_{0}(N) \setminus L_{-N|\Delta|^{2}d^{2},rh}}\chi_{\Delta}(X)\bigg(\sum_{w \in \Q_{<0}}a_{\ell_{X}}^{+}(w)(4\pi w)^{k}e^{2\pi i \Re(c(X))w} \\
	& \qquad \qquad \qquad \qquad +(-1)^{k+1}\sum_{w \in \Q_{<0}}a_{\ell_{-X}}^{+}(w)(4\pi w)^{k}e^{2\pi i \Re(c(-X))w}\bigg).
	\end{align*}
	
	\begin{theorem}\label{thm:fourierexpansiontwisted}
	Let $k \in \Z_{\geq 0}$ and let $F \in H_{-2k}^{+}(\Gamma)$. For $k > 0$ the $h$-th component of $\IM_{\Delta,r}(F,\tau)$ is given by
	\begin{align*}
	&\sum_{m > 0}\frac{1}{2\sqrt{m}}\left(\frac{\sqrt{N}}{4\pi \sqrt{|\Delta|m}}\right)^{k}\big(\mt_{\Delta,r}^{+}(R_{-2k}^{k}F;m,h) + (-1)^{k+1}\mt_{\Delta,r}^{-}(R_{-2k}^{k}F;m,h)\big)q^{m} \\
	&\quad +\sum_{d > 0}\frac{1}{2i\sqrt{N|\Delta|}d}\left(\frac{1}{4\pi i |\Delta|d}\right)^{k}\mt_{\Delta,r}^{c}(F;-N|\Delta|d^{2},h)q^{-N|\Delta|d^{2}} \\
	& \quad +\frac{\sqrt{|\Delta|}(-1)^{k}k!}{2\sqrt{N}\pi^{k+1}}\sum_{\substack{\ell \in \Gamma_{0}(N) \setminus \Iso(V) \\ \ell \cap (L + rh) \neq \emptyset}}a_{\ell}^{+}(0)\frac{\alpha_{\ell}}{\beta_{\ell}^{k+1}}d_{\ell}^{k+1} \\
	& \qquad \qquad \qquad \qquad\qquad \qquad\times\bigg(\sum_{\substack{n > 0 \\ n \equiv m_{\ell} (d_{\ell})}}\frac{\chi_{\Delta}(n)}{n^{s+1}} + (-1)^{k+1}\sgn(\Delta)\sum_{\substack{n > 0 \\ n \equiv -m_{\ell} (d_{\ell})}}\frac{\chi_{\Delta}(n)}{n^{s+1}} \bigg)\bigg|_{s = k} \\
	&\quad-\sum_{m < 0}\frac{1}{2(4\pi|m|)^{k+1/2}|\Delta|^{k/2}}\overline{\mt_{\Delta,r}(\xi_{-2k}F;m,h)}\Gamma\left(\tfrac{1}{2}-k,4\pi|m|v\right)q^{m},
	\end{align*}
	where $m_{\ell},d_{\ell} \in \Z_{\geq 0}$ are defined by $(m_{\ell},d_{\ell}) = 1$ and $k_{\ell}/\beta_{\ell} = m_{\ell}/d_{\ell}$.
	
	For $k = 0$ the $h$-th component of $\IM_{\Delta,r}(F,\tau)$ is given by the same formula as above but with the additional non-holomorphic terms
	\begin{align*}
	&\sum_{d > 0}\frac{1}{4i\sqrt{\pi N|\Delta|}d}\sum_{X \in \Gamma \setminus L_{-N|\Delta|d^{2},rh}}\chi_{\Delta}(X)\big(a_{\ell_{X}}^{+}(0)-a_{\ell_{-X}}^{+}(0)\big)\Gamma\left(\tfrac{1}{2},4\pi N|\Delta|d^{2}v\right)q^{-N|\Delta|d^{2}}.
	\end{align*}
	\end{theorem}

	\begin{proof}
		As in the proof of Proposition \ref{prop:xidiagramtwisted} we write
		\begin{align*}
		\IM_{\Delta,r}(F,\tau) = \frac{1}{[\Gamma_{0}(N):\Gamma_{\Delta}]}\sum_{ h\in L'/L}\left\langle \psi_{\Delta,r}(\e_{h}),\overline{\IM(F,\tau,\dgdelta,\Gamma_{\Delta})}\right\rangle \e_{h}.
		\end{align*}
		We see that the coefficients of the twisted lift can be obtained from the coefficients of the untwisted lift. The twisting of the coefficients of positive and negative index is quite straightforward and can be done as in the proof of Theorem 5.5. in \cite{ae}. 
		
		We sketch the twisting of the constant coefficient. For $h \in \dg$ with $Q(h) \equiv 0(\Z)$ the $(0,h)-$th coefficient of $\IM_{\Delta,r}(F,\tau)$ is given by
		\begin{align*}
	&\frac{\sqrt{|\Delta|}(-1)^{k}k!}{2\sqrt{N}\pi^{k+1}} \frac{1}{[\Gamma_{0}(N):\Gamma_{\Delta}]} \sum_{\substack{\delta \in \dgdelta \\ \pi(\delta) = rh \\ Q_{\Delta}(\delta) \equiv 0 (\Z)} }\chi_{\Delta}(\delta) \\
	&\sum_{\substack{\ell \in \Gamma_{\Delta} \setminus \Iso(V) \\ \ell \cap (\Delta L + \delta) \neq \emptyset}}a_{\ell}^{+}(0)\frac{\alpha_{\ell}^{(\Delta)}}{(\beta_{\ell}^{(\Delta)})^{k+1}}(\zeta(s,k_{\ell}^{(\Delta)}/\beta_{\ell}^{(\Delta)}) + (-1)^{k+1}\zeta(s,1-k_{\ell}^{(\Delta)}/\beta_{\ell}^{(\Delta)}))|_{s = k+1},
	\end{align*}
	where the superscript $(\Delta)$ indicates that the corresponding quantity is taken with respect to the lattice $\Delta L$ with quadratic form $Q_{\Delta}$ and the group $\Gamma_{\Delta}$. It is easy to see that $\beta_{\ell}^{(\Delta)} = |\Delta|\beta_{\ell}$ and $\alpha_{\ell}^{(\Delta)} = [\Gamma_{0}(N)_{\ell}:(\Gamma_{\Delta})_{\ell}]\alpha_{\ell}$, but $k_{\ell}^{(\Delta)}$ is a bit more complicated:
	
	Let $X_{\ell} \in \ell \cap L$ be a positively oriented primitive generator of $\ell$. If $\ell \cap (\Delta L + \delta) \neq \emptyset$ with $\pi(\delta) = rh$ then also $\ell \cap (L + rh) \neq \emptyset$. For a fixed isotropic line $\ell$, a system of representatives for the elements $\delta \in \mathcal{D}(\Delta)$ with $\pi(\delta) = rh, Q_{\Delta}(\delta) \equiv 0(\Z)$ and $\ell \cap (\Delta L + \delta) \neq \emptyset$ is given by the vectors $n X_{\ell} + (rh)_{\ell}$ with $n$ running modulo $|\Delta|$ and some $(rh)_{\ell} \in \ell \cap (L + rh)$. In particular, we have $k_{\ell}^{(\Delta)}/\beta_{\ell}^{(\Delta)} = n/|\Delta| + m_{\ell}/|\Delta|d_{\ell}$. Using the assumption that $\Delta$ is a fundamental discriminant it is not hard to show that $(\Delta,d_{\ell}) = 1$, $\chi_{\Delta}(d_{\ell}) = 1$, and $\chi_{\Delta}(nX_{\ell} + (rh)_{\ell}) = \chi_{\Delta}(nd_{\ell} + m_{\ell})$. Putting everything together, we obtain the twisted constant coefficient.
	\end{proof}
	

	In the same way, we obtain the Fourier expansion of the $(\Delta,r)$-th Shintani lift:

	\begin{theorem}\label{NonvanishingTheorem}
		Let $k \in \Z_{\geq 0}$ and $G \in S_{2k+2}(\Gamma_{0}(N))$. Then the $h$-th component of $\ISh_{\Delta,r}(G,\tau)$ is given by
		\[
		\ISh_{\Delta,r}(G,\tau)_{h} = -\frac{\sqrt{N}}{\sqrt{|\Delta|}}\sum_{m > 0}\frac{1}{|\Delta|^{k/2}}\mt_{\Delta,r}(G;-m,h)q^{m}.
		\]
	\end{theorem}

	\begin{remark}
		Let $N = 1$. In this case the twisted Millson theta function vanishes identically if $(-1)^{k}\Delta > 0$, which easily follows from replacing $X$ by $-X$ in the sum. On the other hand, the results of \cite[Section 5]{ez} show that for $(-1)^{k}\Delta < 0$ the map $f_{0}(\tau)\e_{0} + f_{1}(\tau)\e_{1} \mapsto f_{0}(4\tau)+f_{1}(4\tau)$ defines an isomorphism of $H_{1/2-k,\widetilde{\rho}_{L}}^{+}$ with the subspace of $H_{1/2-k}^{+}(\Gamma_{0}(4))$ of scalar valued harmonic weak Maass forms satisfying the Kohnen plus space condition. Using this identification we can derive the results stated in the introduction from the theorems in this section. Since $\Delta \equiv r^{2}(4)$, $r$ mod $2$ is already determined by $\Delta$, so we can drop it from the notation. The formula for the coefficients of positive index of $\IM_{\Delta}$ follows from $\mt_{\Delta}^{-}(R_{-2k}^{k}F;d) = \sgn(\Delta)\mt_{\Delta}^{+}(R_{-2k}^{k}F;d)$, which can be seen using the map $[a,b,c] \mapsto [-a,b,-c]$, and the formula for the principal part is obtained by rewriting the twisted complementary trace as described in \cite[Proposition 5.7.]{ae}.
	\end{remark}

\begin{remark} As in \cite[Thm 5.1]{alfes} one can show that $\LambdaM_{\Delta,r}$ is orthogonal to cusp forms with respect to the regularized Petersson inner product. In terms of the bilinear pairing $\{\cdot,\cdot\}$ introduced in \cite{brfu04} (or rather its extension given in \cite[Proposition 2.3]{alfes}) this means
\begin{align*}
 \sum_{h\in L'/L}\sum_{\substack{m \in \Q\\m\equiv \sgn(\Delta)Q(h)\, (\Z)}}  c_f^+(-m,h) a^{+}_{\LambdaM_{\Delta,r}(F,\tau)}(m,h) = 0,
\end{align*}
for each $f \in H_{3/2+k,\overline{\widetilde{\rho}}_{L}}^{+}$ with coefficients $c_{f}^{+}(m,h)$ and each $F \in H_{-2k}^{+}(\Gamma_{0}(N))$ such that $\LambdaM_{\Delta,r}(F,\tau)$ is weakly holomorphic. For $N = 1$ and $k = 0$, choosing $F = J = j-744$ and $f=\widetilde\Lambda^{\mathrm{M}}_{\Delta,r}(J,\tau)$, and using that the holomorphic Fourier coefficients of $\widetilde{\Lambda}^{\mathrm{M}}_{\Delta,r}(J,\tau)$ and $\widetilde\Lambda^{\mathrm{M}}_{\Delta,r}(J,\tau)$ are essentially given by the twisted traces of $J$ (compare \cite{ae}), one can recover duality results of Zagier \cite{zatr} for the coefficients of a basis $f_{d} = q^{-d} + O(1)$ for $M_{1/2}^{!,+}(\Gamma_{0}(4))$ and a basis $g_{\Delta} = q^{-\Delta} + O(1)$ for $M_{3/2}^{!,+}(\Gamma_{0}(4))$.
\end{remark}

\subsection{Extensions of the Millson and the Shintani theta lift}\label{Extensions}

In \cite{brfu04} a more general notion of harmonic weak Maass forms is considered. For $k \in \Z$ with $k \neq 1$ and a congruence subgroup $\Gamma$ of $\SL_{2}(\Q)$ the space $H_{k}(\Gamma)$ is defined similarly as the space $H_{k}^{+}(\Gamma)$ but with the growth condition replaced by the weaker requirement that the forms should be at most of linear exponential growth at all cusps. A form $F \in H_{k}(\Gamma)$ has a Fourier expansion with a holomorphic part $F^{+}$ and a non-holomorphic part $F^{-}$,
\begin{align*}
F(z) = F^{+}(z) + F^{-}(z) = \sum_{n \gg -\infty}a^{+}(n)e^{2\pi i n z} + a^{-}(0)y^{1-k}+\sum_{\substack{n \ll \infty \\ n \neq 0}}a^{-}(n)H(2\pi n y)e^{2\pi i nx},
\end{align*} 
where $H(w) = e^{-w}\int_{-2w}^{\infty}e^{-t}t^{-k}dt$, and there are analogous expansions at the other cusps. We consider the subspace $H_{k}^{0}(\Gamma)$ consisting of forms in $H_{k}(\Gamma)$ with vanishing constant terms $a^{-}(0)$ of the non-holomorphic parts at all cusps. It is mapped under $\xi_{k}$ to the space $S_{2-k}^{!}(\Gamma)$ of weakly holomorphic modular forms whose constant terms at all cusps vanish.

The nice observation here is that the proof of Proposition \ref{prop:convergencethetalift}  still goes through for $F \in H_{-2k}^{0}(\Gamma)$. Thus the Millson theta lift of $F \in H_{-2k}^{0}(\Gamma)$ converges to a harmonic function transforming like a modular form of weight $1/2-k$ for $\rho_{L}$. Similarly, the regularization of the Shintani lift also works for a weakly holomorphic modular form $G \in S_{2k+2}^{!}(\Gamma)$ and converges to a harmonic function transforming of weight $3/2+k$ for $\overline{\rho}_{L}$. Further, the relation between the Millson and the Shintani theta lift given in Proposition \ref{prop:xidiagram} still holds for $F \in H_{-2k}^{0}(\Gamma)$, which can be seen by exactly the same proof as for $F \in H_{-2k}^{+}(\Gamma)$.

The computation of the Fourier expansion of $\IM(F,\tau)$ for $F \in H_{-2k}^{0}(\Gamma)$ is almost the same as before, but we have to be careful with the main integral in the computation of the coefficients of negative index since $\xi_{-2k}F$ need no longer be a cusp form. A thorough analysis shows that the non-holomorphic coefficients are now given by traces of regularized cycle integrals of $\xi_{-2k}F$ as introduced in \cite{brifrike}, and that there is now also a contribution of the coefficients $a_{\ell}^{-}(w)$ for $w > 0$ to the complementary trace. Similarly, the coefficients of the Shintani lift of $G \in S_{2k+2}^{!}(\Gamma)$ are given by traces of regularized cycle integrals of $G$ as in \cite{briguka}.

The twisting of these extended lifts proceeds in the same way as before. 
We obtain the following extension of (the twisted versions of) Proposition \ref{prop:xidiagram} and Theorem \ref{thm:PropertiesLifts}.

\begin{theorem}
		Let $k \in \Z_{\geq 0}$.
		\begin{enumerate}
			\item The Millson theta lift $\IM_{\Delta,r}$ maps $H^{0}_{-2k}(\Gamma_{0}(N))$ to $H^{+}_{1/2-k,\widetilde{\rho}_{L}}$.
			\item The Shintani theta lift $\ISh_{\Delta,r}$ maps $S^{!}_{2k+2}(\Gamma_{0}(N))$ to $S_{3/2+k,\overline{\widetilde{\rho}}_{L}}$.
			\item The relation between the Millson and the Shintani theta lift given in Proposition \ref{prop:xidiagramtwisted} also holds for $F \in H^{0}_{-2k}(\Gamma_{0}(N))$.
		\end{enumerate}
	\end{theorem}

	\section{Cycle integrals}
	
	In this section we prove some identities between the cycle integrals of $R_{-2k}^{2j+1}F, j \geq 0,$ and $\xi_{-2k}F$ for a harmonic weak Maass form $F \in H_{-2k}^{+}(\Gamma)$, where $\Gamma$ is some congruence subgroup of $\SL_{2}(\Q)$ again.
	
	\subsection{Closed geodesics}
		
	Let $X \in V$ with $Q(X) = m < 0$ such that $|m|/N$ is not a square in $\Q$, i.e.~the stabilizer $\overline{\Gamma}_{X}$ is infinite cyclic and $c(X) = \Gamma_{X}\setminus c_{X}$ is a closed geodesic. Further, let $G$ be some smooth function that transforms like a modular form of weight $2k+2$ under $\Gamma$ for some $k \in \Z$. Recall the definition of the cycle integral
	\[
	\mathcal{C}(G,X) = (-2\sqrt{|m|N}i)^{k}i\int_{1}^{\varepsilon^{2}}G_{g}(iy)y^{k}dy,
	\]
	where $g \in \SL_{2}(\R)$ is such that $g^{-1}Xg = \sqrt{|m|/N}\left(\begin{smallmatrix} 1 & 0 \\ 0 & - 1 \end{smallmatrix} \right)$, $\varepsilon > 1$ is such that $\left( \begin{smallmatrix}\varepsilon & 0 \\ 0 & \varepsilon^{-1}\end{smallmatrix}\right)$ generates the stabilizer of $g^{-1}Xg$ in $g^{-1}\overline{\Gamma}g$, and $G_{g} = G|_{2k+2}g$.

	\begin{proposition}\label{CycleIntegralInduction}
		Let $X \in V$ with $Q(X) = m < 0$ such that $|m|/N$ is not a square. Let $k \in \Z$ and $F \in H_{-2k}^{+}(\Gamma)$. For all integers $\ell \leq k$ we have
		\begin{align}\label{CycleInductionStart}
		\mathcal{C}(R_{-2k}^{k-\ell+1}F,X) = \frac{1}{(4|m|N)^{\ell}}\overline{\mathcal{C}(\xi_{-2\ell}R_{-2k}^{k-\ell}F,X)}.
		\end{align}
		Further, for $\ell \leq k-1$ we have
		\begin{align}\label{CycleInductionStepR}
		\mathcal{C}(R_{-2k}^{k-\ell+1}F,X) = 4|m|N(k-\ell)(k+\ell+1)\mathcal{C}(R_{-2k}^{k-\ell-1}F,X).
		\end{align}
	\end{proposition}
	
	\begin{proof}
		Plugging in the definition of the cycle integral, the left-hand side of (\ref{CycleInductionStart}) equals
		\[
		(-2\sqrt{|m|N}i)^{-\ell}i\int_{1}^{\varepsilon^{2}}(R_{-2k}^{k-\ell+1}F_{g})(iy)y^{-\ell}dy.
		\] 
		Since $\ell \leq k$ we can split off the outermost raising operator $R_{-2\ell} = 2i\frac{\partial}{\partial z} - 2\ell y^{-1}$ to obtain
		\[
		(R_{-2k}^{k-\ell+1}F_{g})(iy)y^{-\ell} = 2i\left(\frac{\partial}{\partial z}R_{-2k}^{k-\ell}F_{g}\right)(iy)y^{-\ell} - 2\ell(R_{-2k}^{k-\ell}F_{g})(iy)y^{-\ell-1}.
		\]
		Now we use $\frac{\partial}{\partial z} = \frac{\partial}{\partial \bar{z}} - i\frac{\partial}{\partial y}$ and apply the product rule to the $\frac{\partial}{\partial y}$-part to get
		\[
		(R_{-2k}^{k-\ell+1}F_{g})(iy)y^{-\ell} = 2i\left(\frac{\partial}{\partial \bar{z}}R_{-2k}^{k-\ell}F_{g}\right)(iy)y^{-\ell} + 2\frac{\partial}{\partial y}\left((R_{-2k}^{k-\ell}F_{g})(iy)y^{-\ell}\right).
		\]
		Note that we also used $(\frac{\partial}{\partial y}R_{-2k}^{k-\ell}F_{g})(iy) = \frac{\partial}{\partial y}((R_{-2k}^{k-\ell}F_{g})(iy))$.
		The first summand on the right-hand side equals
		\[
		2i\left(\frac{\partial}{\partial \bar{z}}R_{-2k}^{k-\ell}F_{g}\right)(iy)y^{-\ell} = -\overline{(\xi_{-2\ell}R_{-2k}^{k-\ell}F_{g})(iy)}y^{\ell},
		\]
		giving the right-hand side of (\ref{CycleInductionStart}). Further, the integral
		\[
		\int_{1}^{\varepsilon^{2}}\frac{\partial}{\partial y}\left((R_{-2k}^{k-\ell}F_{g})(iy)y^{-\ell}\right)dy =  (R_{-2k}^{k-\ell}F_{g})(i\varepsilon^{2})\varepsilon^{-2\ell}-(R_{-2k}^{k-\ell}F_{g})(i)
		\]
		vanishes since $(R_{-2k}^{k-\ell}F_{g})(i\varepsilon^{2})\varepsilon^{-2\ell} = (R_{-2k}^{k-\ell}F_{g})|_{-2\ell}\left( \begin{smallmatrix}\varepsilon & 0 \\ 0 & \varepsilon^{-1} \end{smallmatrix}\right)(i)$ and $R_{-2k}^{k-\ell}F_{g}$ transforms like a modular form of weight $-2\ell$ for $g^{-1}\Gamma g$. This completes the proof of (\ref{CycleInductionStart}).
		
		The formula (\ref{CycleInductionStepR}) easily follows from (\ref{CycleInductionStart}) if we use that 
		\[
		\overline{\xi_{-2\ell}R_{-2k}^{k-\ell}F} = (k-\ell)(k+\ell+1)y^{-2\ell-2}R_{-2k}^{k-\ell-1}F
		\]
		for all $k \in \Z$, all integers $\ell \leq k-1$ and $F \in H_{-2k}^{+}(\Gamma)$. This follows from Lemma \ref{lm:reldiff} if we write $\xi_{-2\ell} = y^{-2\ell-2}\overline{L_{-2\ell}}$ and use the relation (\ref{eq:DkLkRk}).
	\end{proof}
	
	\begin{corollary}\label{CycleIntegralMainTheorem}
		Let $X \in V$ with $Q(X) = m < 0$ such that $|m|/N$ is not a square. Further, let $k\in \Z_{\geq 0}$ and $F \in H_{-2k}^{+}(\Gamma)$. For $j \in \Z_{\geq 0}$ we have
		\begin{align*}
		\mathcal{C}(R_{-2k}^{2j+1}F,X) = \frac{1}{(4|m|N)^{k-j}}\frac{j!(k-j)!(2k)!}{k!(2k-2j)!}\overline{\mathcal{C}(\xi_{-2k}F,X)}.
		\end{align*}
	\end{corollary}

	\begin{proof}
	We use (\ref{CycleInductionStart}) with $\ell = k$ and then repeatedly apply (\ref{CycleInductionStepR}).
	\end{proof}
	
	As a we special case we obtain a generalization of Theorem 1.1. from \cite{briguka2}. 
	
	\begin{corollary}\label{CycleIntegralTheorem}
		Let $X \in V$ with $Q(X) = m < 0$ such that $|m|/N$ is not a square. Further, let $k\in \Z_{\geq 0}$ and $F \in H_{-2k}^{+}(\Gamma)$. For even $k$ we have
		\begin{align*}
		\mathcal{C}(R_{-2k}^{k+1}F,X) = \frac{1}{(4|m|N)^{k/2}}\frac{((\tfrac{k}{2})!)^{2}(2k)!}{(k!)^{2}}\overline{\mathcal{C}(\xi_{-2k}F,X)},
		\end{align*}
		and for odd $k$ we have
		\begin{align*}
		\mathcal{C}(R_{-2k}^{k}F,X) = \frac{1}{(4|m|N)^{(k+1)/2}}\frac{(\tfrac{k-1}{2})!(\tfrac{k+1}{2})!(2k)!}{(k+1)!k!}\overline{\mathcal{C}(\xi_{-2k}F,X)}.
		\end{align*}
	\end{corollary}

	Moreover, we obtain the non-square part of Theorem 1.1 from \cite{briguka} which asserts that the cycle integrals of the weight $2k+2$ weakly holomorphic modular forms $\mathcal{D}^{2k+1}F = -(4\pi)^{-(2k+1)}R_{-2k}^{2k+1}F$ and $\xi_{-2k}F$ agree up to some constant.
	
	\begin{corollary}\label{CycleIntegralTheorem2}
		Let $X \in V$ with $Q(X) = m < 0$ such that $|m|/N$ is not a square.	For $k \in \Z_{\geq 0}$ and $F \in H_{-2k}^{+}(\Gamma)$ we have
		\[
		\mathcal{C}(\mathcal{D}^{2k+1}F,X) = -\frac{(2k)!}{(4\pi)^{2k+1}}\overline{\mathcal{C}(\xi_{-2k}F,X)}.
		\]
	\end{corollary}

	\subsection{Infinite geodesics}\label{InfiniteGeodesics}
	
	Let $X \in V$ with $Q(X) = m < 0$ such that $|m|/N$ is a square in $\Q$, i.e.~the stabilizer $\overline{\Gamma}_{X}$ is trivial and $c(X) = \Gamma_{X}\setminus c_{X}$ is an infinite geodesic in $\Gamma \setminus \h$. Recall that for a cusp form $G \in S_{2k+2}$ the cycle integral is defined by
	\[
	\mathcal{C}(G,X) = (-2\sqrt{|m|N}i)^{k}i\int_{0}^{\infty}G_{g}(iy)y^{k}dy,
	\]
	where $g \in \SL_{2}(\R)$ is such that $g^{-1}Xg = \sqrt{|m|/N}\left(\begin{smallmatrix} 1 & 0 \\ 0 & - 1 \end{smallmatrix} \right)$ and $G_{g} = G|_{2k+2}g$. 
	
	We would like to prove similar identities as in the last section, but in general the cycle integral of $R_{-2k}^{k-\ell}F$ does not converge if the geodesic is infinite. If we start with the (convergent) cycle integral of $\xi_{-2k}F$ and repeat the calculations of the last section, we are led to suitable regularized cycle integrals of $R_{-2k}^{k-\ell}F$.
	
	First, for $F \in H_{-2k}^{+}(\Gamma)$ we write
	\[
	\mathcal{C}(\xi_{-2k}F,X) = (-2\sqrt{|m|N}i)^{k}i\left(\int_{1}^{\infty}\xi_{-2k}F_{g}(iy)y^{k}dy + (-1)^{k+1}\int_{1}^{\infty}\xi_{-2k}F_{gS}(iy)y^{k}dy \right),
	\]
	with $S = \left(\begin{smallmatrix}0 & -1 \\ 1 & 0 \end{smallmatrix}\right)$, where we split the integral over $(0,\infty)$ at $1$ and replaced $y$ by $1/y$ in the integral over $(0,1)$. Now we decompose $F_{g} = F_{g}^{+} + F_{g}^{-}$ and $F_{gS} = F_{gS}^{+} + F_{gS}^{-}$ into their holomorphic and non-holomorphic parts and use $\xi_{-2k}F_{g} = \xi_{-2k}F_{g}^{-}$ and $\xi_{-2k}F_{gS} = \xi_{-2k}F_{gS}^{-}$. Note that $F_{g}^{-}$ and $F_{gS}^{-}$ are rapidly decreasing at the cusp $\infty$, but not necessarily at $0$, and this is the reason why we split the integral above. We have the following analog of Proposition \ref{CycleIntegralInduction}:
	
	\begin{proposition}
		Let $X \in V$ with $Q(X) = m < 0$ such that $|m|/N$ is a square. Let $k \in \Z$ and $F \in H_{-2k}^{+}(\Gamma)$. For all integers $\ell \leq k$ we have
	\begin{align*}
	\int_{1}^{\infty}R_{-2k}^{k-\ell+1}F_{g}^{-}(iy)y^{-\ell}dy = -\overline{\int_{1}^{\infty}\xi_{-2\ell}R_{-2k}^{k-\ell}F_{g}^{-}(iy)y^{\ell} dy} - 2R_{-2k}^{k-\ell}F_{g}^{-}(i).
	\end{align*}
		Further, for $\ell \leq k-1$ we have
		\begin{align*}
		\int_{1}^{\infty}R_{-2k}^{k-\ell+1}F_{g}^{-}(iy)y^{-\ell}dy = -(k-\ell)(k+\ell+1)\int_{1}^{\infty}R_{-2k}^{k-\ell-1}F_{g}^{-}(iy)y^{-\ell-2}dy-2R_{-2k}^{k-\ell}F_{g}^{-}(i).
		\end{align*}
	The same formulas hold with $g$ replaced by $gS$.
	\end{proposition}
	
	\begin{proof}
		The computations are the same as in the proof of Proposition \ref{CycleIntegralInduction} if we replace $\varepsilon^{2}$ by $\infty$ and use the rapid decay of $R_{-2k}^{k-\ell}F^{-}(iy)$ as $y \to \infty$.
	\end{proof}
	
	A repeated application of the proposition leads to following definition: For every integer $j \geq 0$ we define the regularized cycle integral of $R_{-2k}^{2j+1}F$ by
	\begin{align*}
	&\mathcal{C}^{\reg}(R_{-2k}^{2j+1}F,X)=(-2\sqrt{|m|N}i)^{-k+2j}i \\
	&\quad\quad\quad\times\bigg(\sum_{\ell=0}^{j}C_{\ell,j}(R_{-2k}^{2\ell}F^{-}_{g}(i)) + (-1)^{k+1}\sum_{\ell=0}^{j}C_{\ell,j}(R_{-2k}^{2\ell}F^{-}_{gS}(i)) \\ 
	&\quad \qquad \qquad + \int_{1}^{\infty}R_{-2k}^{2j+1}F_{g}^{-}(iy)y^{-k+2j}dy+(-1)^{k+1}\int_{1}^{\infty}R_{-2k}^{2j+1}F_{gS}^{-}(iy)y^{-k+2j}dy\bigg),
	\end{align*}
	where $C_{\ell,j} = 2(-1)^{\ell+j}\prod_{t=\ell+1}^{j}(2t)(2k-2t+1)$. Note that
	\[
	R_{-2k}^{2\ell}F^{-}_{g}(i) + (-1)^{k+1}R_{-2k}^{2\ell}F^{-}_{gS}(i) = -R_{-2k}^{2\ell}F^{+}_{g}(i) - (-1)^{k+1}R_{-2k}^{2\ell}F^{+}_{gS}(i),
	\]
	so the second line above can also be understood as the part of the regularized cycle integral coming from $F^{+}$.
	
	With this definition, we find
	\begin{align*}
	\mathcal{C}^{\reg}(R_{-2k}F,X) = \frac{1}{(4|m|N)^{k}}\overline{\mathcal{C}(\xi_{-2k}F,X)}
	\end{align*}
	and
	\begin{align*}
	\mathcal{C}^{\reg}(R_{-2k}^{2j+1}F,X) = 4|m|N(2j)(2k-2j+1)\mathcal{C}^{\reg}(R_{-2k}^{2j-1}F,X)
	\end{align*}
	for $j \geq 1$, and thus all the corollaries of the last section also hold for $|m|/N$ being a square. 
	
	\begin{remark}
	For $k = j = 0$ and $F \in H_{0}^{+}(\Gamma)$ the regularized cycle integral of  $R_{0}F$ is defined by
	\begin{align*}
	\mathcal{C}^{\reg}(R_{0}F,X) &= 2iF^{-}_{g}(i) - 2iF^{-}_{gS}(i) + i\int_{1}^{\infty}R_{0}F^{-}_{g}(iy)dy - \int_{1}^{\infty}R_{0}F_{gS}(iy)dy.
	\end{align*}
	On the other hand, since $R_{0}F \in S_{2}^{!}(\Gamma)$ is in fact a weakly holomorphic cusp form, there is a regularized cycle integral studied in \cite{brifrike},\cite{briguka2} and \cite{briguka}. It is given by
	\[
	\mathcal{C}_{\text{BFK}}^{\reg}(R_{0}F,X) = \left[ i\int_{1}^{\infty}R_{0}F_{g}(iy)e^{-ys}dy\right]\bigg|_{s = 0} - \left[ i\int_{1}^{\infty}R_{0}F_{gS}(iy)e^{-ys}dy\right]\bigg|_{s = 0},
	\]
	where the expression on the right means that one has to take the value at $s = 0$ of the analytic continuation of the integral. We want to compare the two regularizations. Let us split $F_{g} = F_{g}^{+} + F^{-}_{g}$. Due to the rapid decay of $F_{g}^{-}$, we can plug in $s = 0$ in the integral over $F_{g}^{-}$. In the integral over $F_{g}^{+}$, we insert the Fourier expansion $F_{g}^{+}(z) = \sum_{n \in \Q}a_{g}^{+}(n)e^{2\pi i nz}$, apply $R_{0} = 2i\frac{\partial}{\partial z}$ and obtain after a short calculation
	\begin{align*}
	\bigg[i\int_{1}^{\infty}R_{0}F_{g}^{+}(iy)e^{-ys}dy\bigg]\bigg|_{s = 0} = \bigg[-4\pi i\sum_{n \neq 0}\frac{na^{+}_{g}(n)}{2\pi n +s}e^{-(2\pi n+s)}\bigg]\bigg|_{s = 0} = -2iF_{g}^{+}(i)+2ia^{+}_{g}(0).
	\end{align*}
	Using $F_{g}^{+}(i)-F_{gS}^{+}(i) = -F_{g}^{-}(i)+F_{gS}^{-}(i)$ we find
	\[
	\mathcal{C}^{\reg}(R_{0}F,X)  = \mathcal{C}_{\text{BFK}}^{\reg}(R_{0}F,X)  -2ia^{+}_{g}(0) + 2ia^{+}_{gS}(0).
	\]
	Note that the regularized cycle integrals considered in \cite{brifrike} are only studied for weakly holomorphic cusp forms, and the analytic continuation of the integrals relies on the particular shape of the Fourier expansion of such forms.
	For general $k$ and $j$, the function $R_{-2k}^{2j+1}F$ is not weakly holomorphic and has a somewhat complicated Fourier expansion, so it is not clear that the regularization of \cite{brifrike} works. It would be interesting to investigate this problem in the future.
	
	Finally, we remark that our regularized cycle integrals look very similar to the cycle integrals of weight zero harmonic weak Maass forms given in \cite{bif}. However, the definitions do not overlap since we only consider cycle integrals of $R_{-2k}^{\ell}F$ for odd $\ell$. Again, it would be nice to unify the approaches and define regularized cycle integrals of $R_{-2k}^{\ell}F$ for all $\ell \geq 0$.
	\end{remark}
	

%

\bibliographystyle{alpha}
\bibliography{bib.bib}

\end{document}